\documentclass{amsart}

\usepackage{bm}
\usepackage{graphicx,here}

\newtheorem{theo}{Theorem}[section]
\newtheorem{prop}[theo]{Proposition}
\newtheorem{lem}[theo]{Lemma}
\newtheorem{theorem}[theo]{Theorem}
\theoremstyle{definition}

\theoremstyle{remark}
\newtheorem{remark}[theo]{Remark}

\begin{document}

\title{Minimal coloring numbers on minimal diagrams of torus links}

\author{Kazuhiro Ichihara}
\address{Department of Mathematics, College of Humanities and Sciences, Nihon University, 3-25-40 Sakurajosui, Setagaya-ku, Tokyo 156-8550, Japan}
\email{ichihara.kazuhiro@nihon-u.ac.jp}

\author{Katsumi Ishikawa}
\address{Research Institute for Mathematical Sciences, Kyoto University, Kyoto 606-8502, Japan}
\email{katsumi@kurims.kyoto-u.ac.jp}

\author{Eri Matsudo}
\address{Institute of Natural Sciences, 
Nihon University, 3-25-40 Sakurajosui, Setagaya-ku, Tokyo 156-8550, Japan}
\email{cher16001@g.nihon-u.ac.jp}

\thanks{The first author is partially supported by JSPS KAKENHI Grant Number JP18K03287.}
\thanks{The second author was partially supported by JSPS KAKENHI Grant Number JP16J01183.}

\date{\today}

\begin{abstract}
We determine the minimal number of colors for non-trivial $\mathbb{Z}$-colorings on the standard minimal diagrams of $\mathbb{Z}$-colorable torus links. 
Also included are complete classifications of such $\mathbb{Z}$-colorings and of such $\mathbb{Z}$-colorings by only four colors, which are shown by using rack colorings on link diagrams. 
\end{abstract}

\maketitle

\section{Introduction}

This is a continuation of the study of $\mathbb{Z}$-colorings on the standard minimal diagrams of $\mathbb{Z}$-colorable torus links given in \cite[Section 3]{IchiharaMatsudo3} by the first and the third named authors. 

Previously, in \cite{IchiharaMatsudo2}, as a generalization to the well-known Fox's coloring originally introduced in \cite{Fox}, they defined a $\mathbb{Z}$-coloring for a link in the 3-space as follows. 
A map $\gamma:\{$arcs of $D\}\rightarrow \mathbb{Z}$ for a regular diagram $D$ of a link is called a \textit{$\mathbb{Z}$-coloring} if it satisfies the condition $2\gamma(a)= \gamma(b)+\gamma(c)$ 
at each crossing of $D$ with the over arc $a$ and the under arcs $b$ and $c$. 
We say that a link is \textit{$\mathbb{Z}$-colorable} if it has a diagram admitting a non-trivial $\mathbb{Z}$-coloring, i.e., there are at least two distinct colors on the diagram. 
We remark that a link $L$ is $\mathbb{Z}$-colorable if and only if the determinant $\det(L)$ of $L$ equals $0$. 
Since the determinant of any knot (single component link) is shown to be an odd integer, any knot is not $\mathbb{Z}$-colorable. 

There are several studies on the \textit{minimal coloring number} (i.e., the minimal number of the colors) of Fox colorings on knots and links; some upper and lower bounds have been obtained. 
In the same line, it was shown in \cite[Theorem 3.1]{IchiharaMatsudo2} that the minimal coloring number $mincol_{\mathbb{Z}}(L)$ of a non-splittable $\mathbb{Z}$-colorable link $L$ is at least four. 
(Note that $mincol_{\mathbb{Z}}(L) = 2$ for any splittable link $L$.) 
However, in contrast to the case of the Fox coloring, it was proved that $mincol_{\mathbb{Z}}(L) = 4$ for any non-splittable $\mathbb{Z}$-colorable link $L$, by the third author in \cite{Matsudo}, and independently by Meiqiao Zhang, Xian'an Jin, and Qingying Deng in \cite{ZhangJinDeng}. 

Thereafter, in \cite{IchiharaMatsudo3}, the first and third authors consider the minimal coloring number $mincol_\mathbb{Z}(D)$ of a minimal diagram $D$ of a $\mathbb{Z}$-colorable link. (A \textit{minimal diagram} is a diagram representing the link with least number of crossings.) 

In particular, in \cite[Section 3]{IchiharaMatsudo3}, they consider \textit{torus links}, that is, the links which can be isotoped onto the standardly embedded torus in the $3$-space. 
By $T(a,b)$, we mean the torus link running $a$ times meridionally and $b$ times longitudinally. 
It is known that $T(a,b)$ is $\mathbb{Z}$-colorable if $a$ or $b$ is even. 
Actually, it is shown in \cite[Theorem 1.3]{IchiharaMatsudo3} that $mincol_\mathbb{Z}(D) = 4$ for the standard diagram $D$ of $T(pn,n)$ illustrated by Figure~\ref{torus} with $n >2$, even and $p \ne 0$. 
Extending this, in this paper, we show the following: 

\begin{theorem}\label{main-thm}
Let $p, q,$ and $r$ be non-zero integers such that $p$ and $q$ are relatively prime, $|p|\geq q\geq 1,$ and $r\geq 2$. 
Let $D$ be the standard diagram of $T(pr,qr)$ illustrated by Figure~\ref{torus}. 
Suppose that $T(pr,qr)$ is $\mathbb{Z}$-colorable, or, equivalently, $pr$ or $qr$ is even. 
Then, $mincol_\mathbb{Z}(D) =4$ if $r$ is even, and $mincol_\mathbb{Z}(D)=5$ if $r$ is odd. 
\end{theorem}

\begin{figure}[H]
\begin{center}
\includegraphics[width=.7\textwidth]{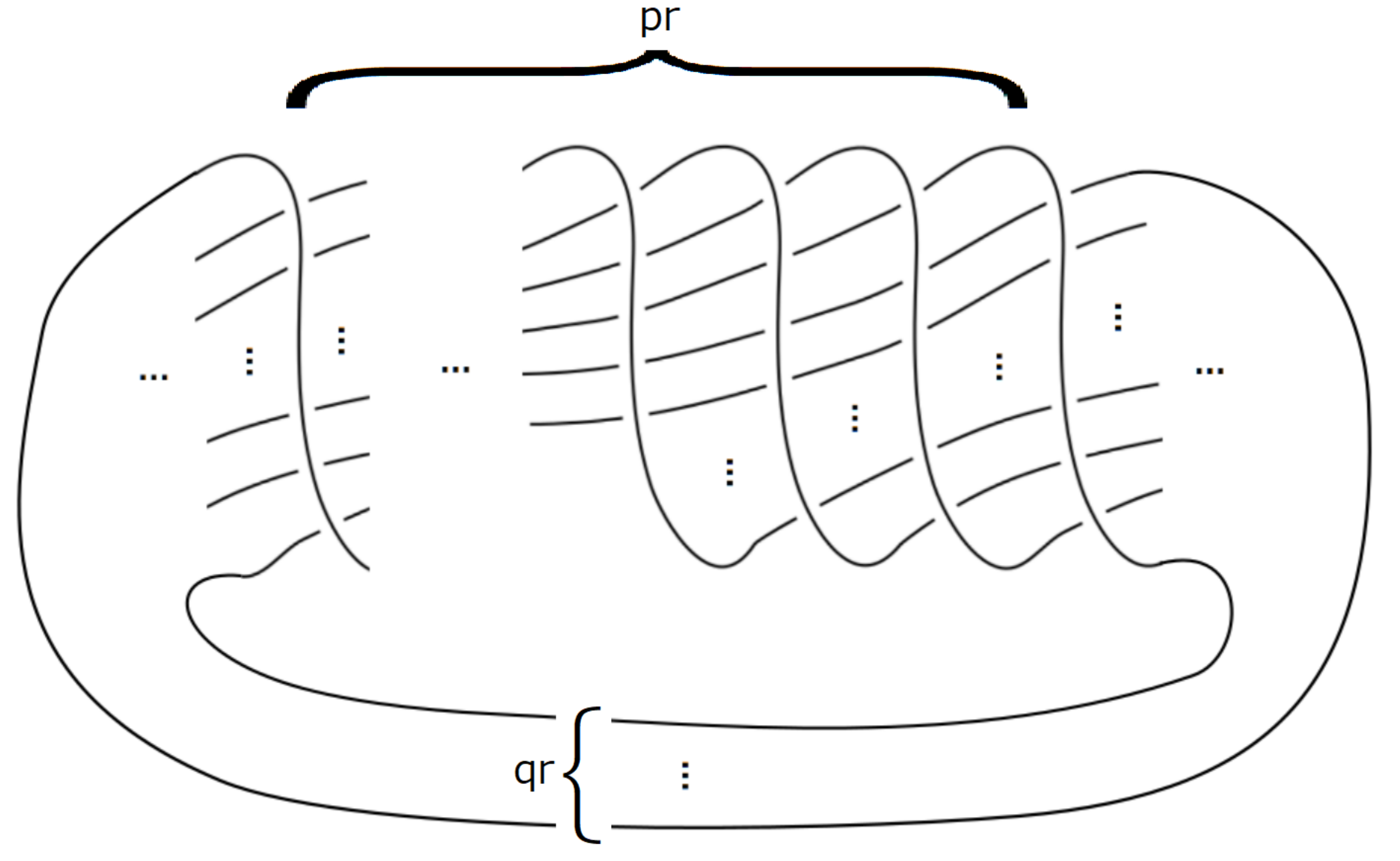}
\caption{The standard diagram of $T(pr,qr)$}\label{torus}
\end{center}
\end{figure}

We remark that the diagram $D$ has the least number of crossings for the torus link. See \cite{Kawauchi} for example. 
Also if $r=1$, then the link becomes a knot, which is not $\mathbb{Z}$-colorable. 
(Actually $r$ coincides with the number of components of the torus link.) 
The theorem above is proved in Section~\ref{sec:FourColors} (when $r$ is even) and Section~\ref{sec:FiveColors} (when $r$ is odd). 

We also include complete classifications of all the $\mathbb{Z}$-colorings on the standard diagram of $T(a,b)$ (Proposition~\ref{col-prop}) and of all the $\mathbb{Z}$-colorings by only four colors of $T(a,b)$ (Proposition~\ref{4col-prop}) in Section~\ref{sec:description}. 
The key of our proof of the propositions is to use rack colorings on link diagrams. 
A theorem used to prove the propositions, which can be of interest independently, is proved in Appendix. 
That part is essentially based on the master thesis \cite{Ishikawa} of the second author.


\section{Four colors case}\label{sec:FourColors}

In this section, we prove the following, showing the first assertion of Theorem~\ref{main-thm}. 

\begin{theorem}\label{thm_r:even}
Let $p, q,$ and $r$ be non-zero integers such that $p$ and $q$ are relatively prime, $|p|\geq q\geq 1,$ and $r\geq 2$. 
Let $D$ be the standard diagram of $T(pr,qr)$ illustrated by Figure~\ref{torus}. 
Suppose that $T(pr,qr)$ is $\mathbb{Z}$-colorable, or, equivalently, $pr$ or $qr$ is even. 
Then, $mincol_\mathbb{Z}(D) =4$ if $r$ is even. 
\end{theorem}

\begin{proof}
We will find a $\mathbb{Z}$-coloring $\gamma$ with only four colors on $D$ by assigning colors on the arcs of $D$. 

Note that the link has $r$ components each of which runs longitudinally $q$ times and twists meridionally $p$ times as shown in Figure \ref{torus}. 
In a local view, we see $qr$ horizontal parallel arcs in $D$. 
We divide such arcs into $q$ subfamilies $\mathbf{x}_1, \dots, \mathbf{x}_q$ as depicted in Figure \ref{torus9-1}~(left).

\begin{figure}[H]
\begin{center}
\includegraphics[width=.43\textwidth]{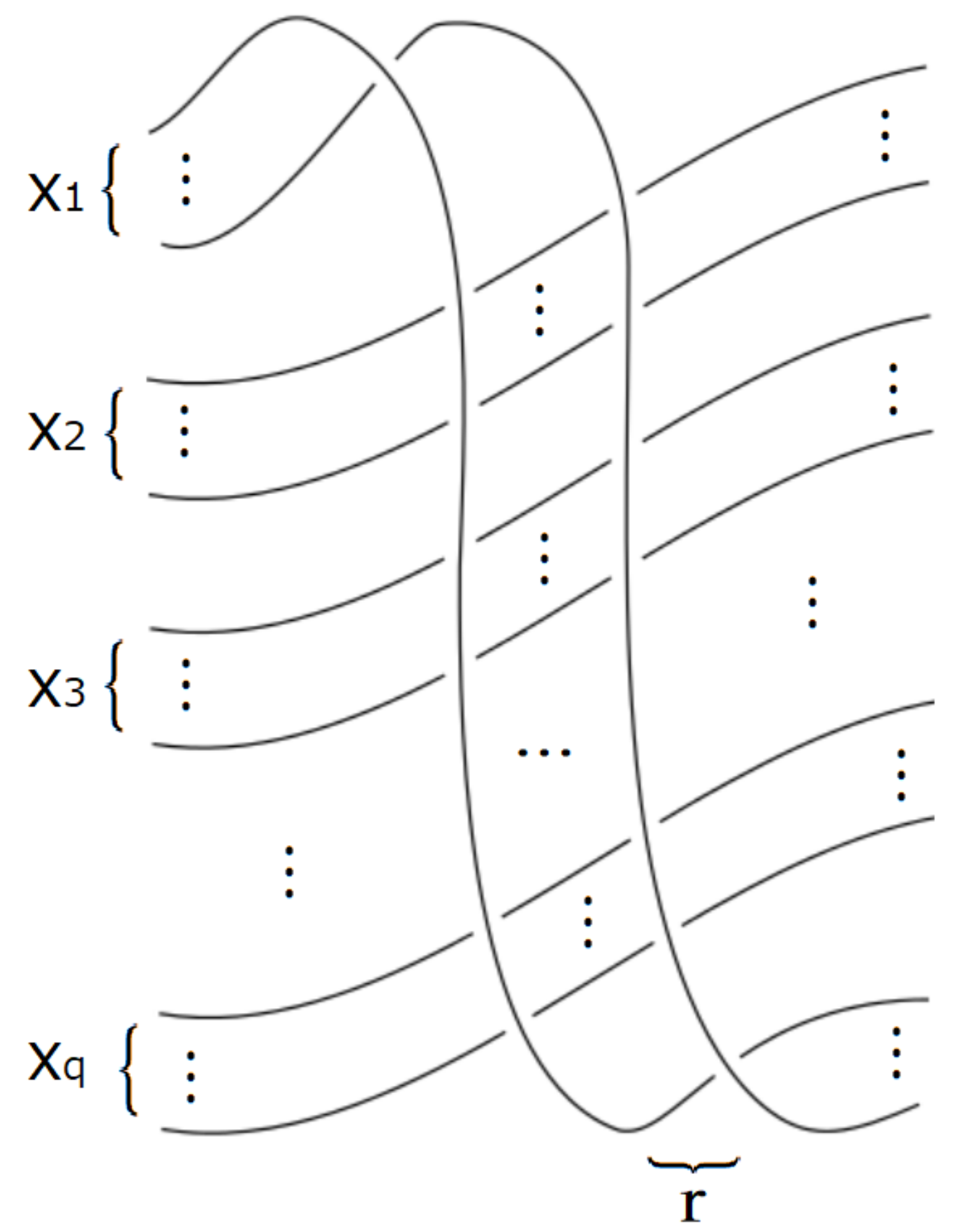}
\qquad 
\includegraphics[width=.43\textwidth]{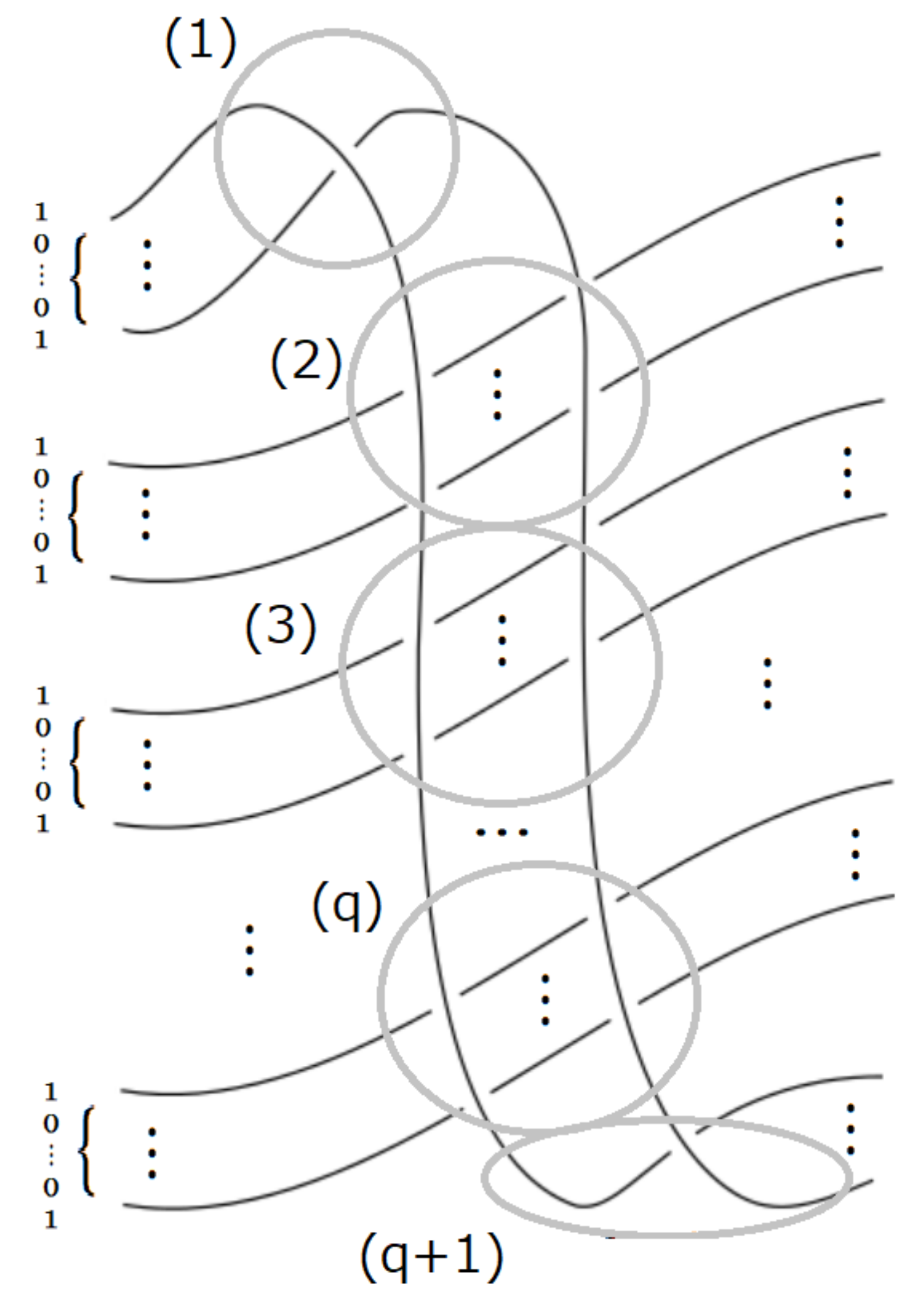}
\caption{}\label{torus9-1}
\end{center}
\end{figure}

We first find a local $\mathbb{Z}$-coloring $\gamma$ on the local diagram shown in Figure \ref{torus9-1}~(left). 
Let us start with setting $(\gamma(x_{i,1}),\gamma(x_{i,2}),\dots ,\gamma(x_{i,r})) = (1, 0, \dots, 0, 1)$ for any $i$. 
See Figure \ref{torus9-1}~(right).

Since $r$ is assumed to be even, as illustrated in Figures \ref{torus9-2}~(left), we can extend $\gamma$ on the arcs in the regions $(1)$ and $(q+1)$ in Figure \ref{torus9-1}~(right), and, as illustrated in Figures \ref{torus9-2}~(right), we can extend $\gamma$ on the arcs in the regions $(2), (3), \dots, (q)$ in Figure \ref{torus9-1}~(right). 
Thus $\gamma$ can be extended on all the arcs in the local diagram as depicted in Figure \ref{torus9-1-2}.

\begin{figure}[H]
\begin{center}
\includegraphics[width=.575\textwidth]{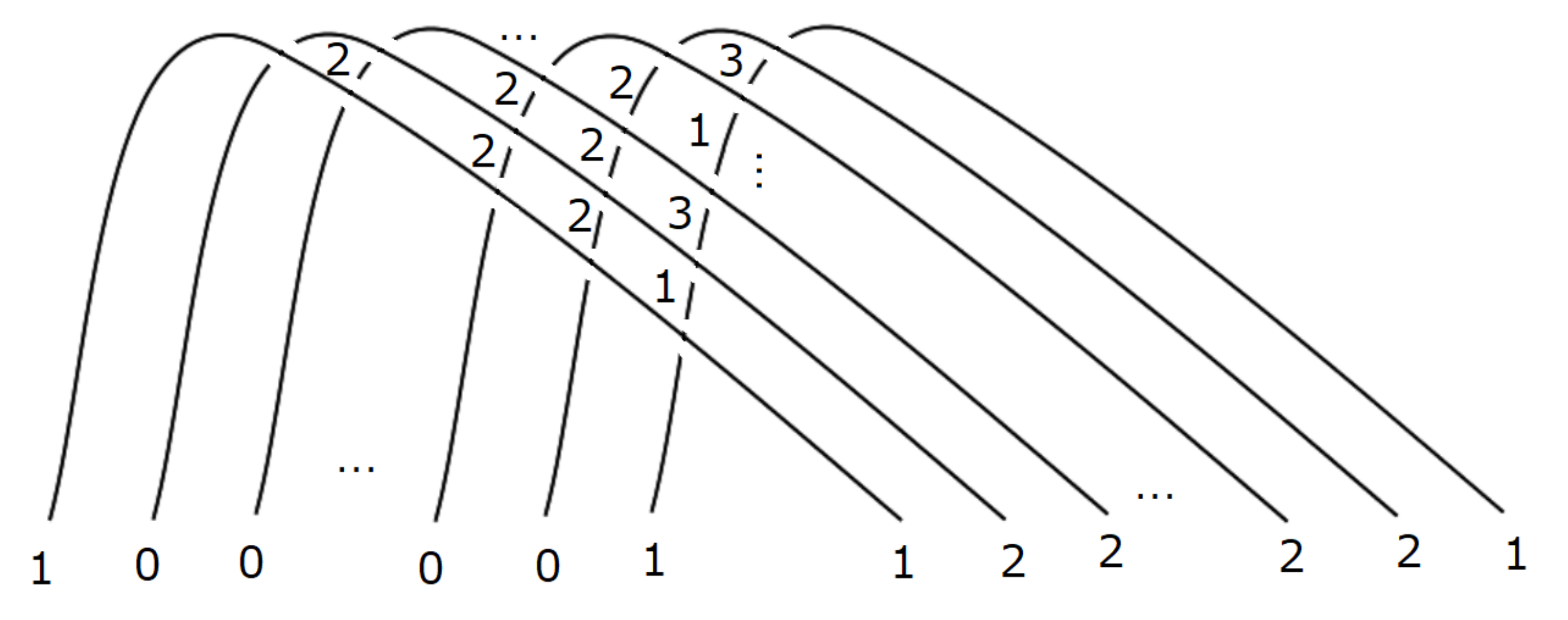}
\quad 
\includegraphics[width=.375\textwidth]{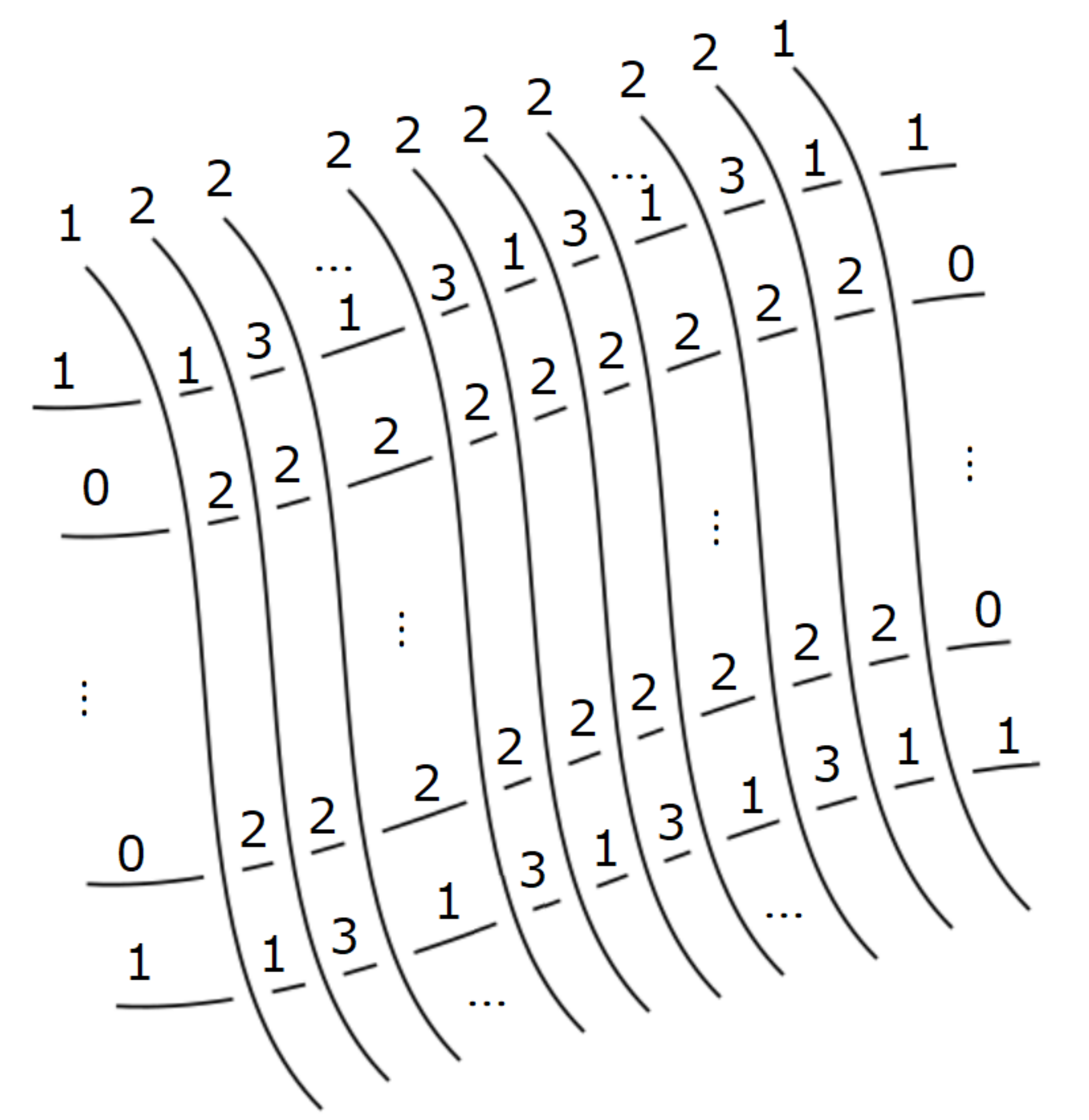}
\caption{}\label{torus9-2}
\end{center}
\end{figure}

\begin{figure}[H]
\begin{center}
\includegraphics[width=.45\textwidth]{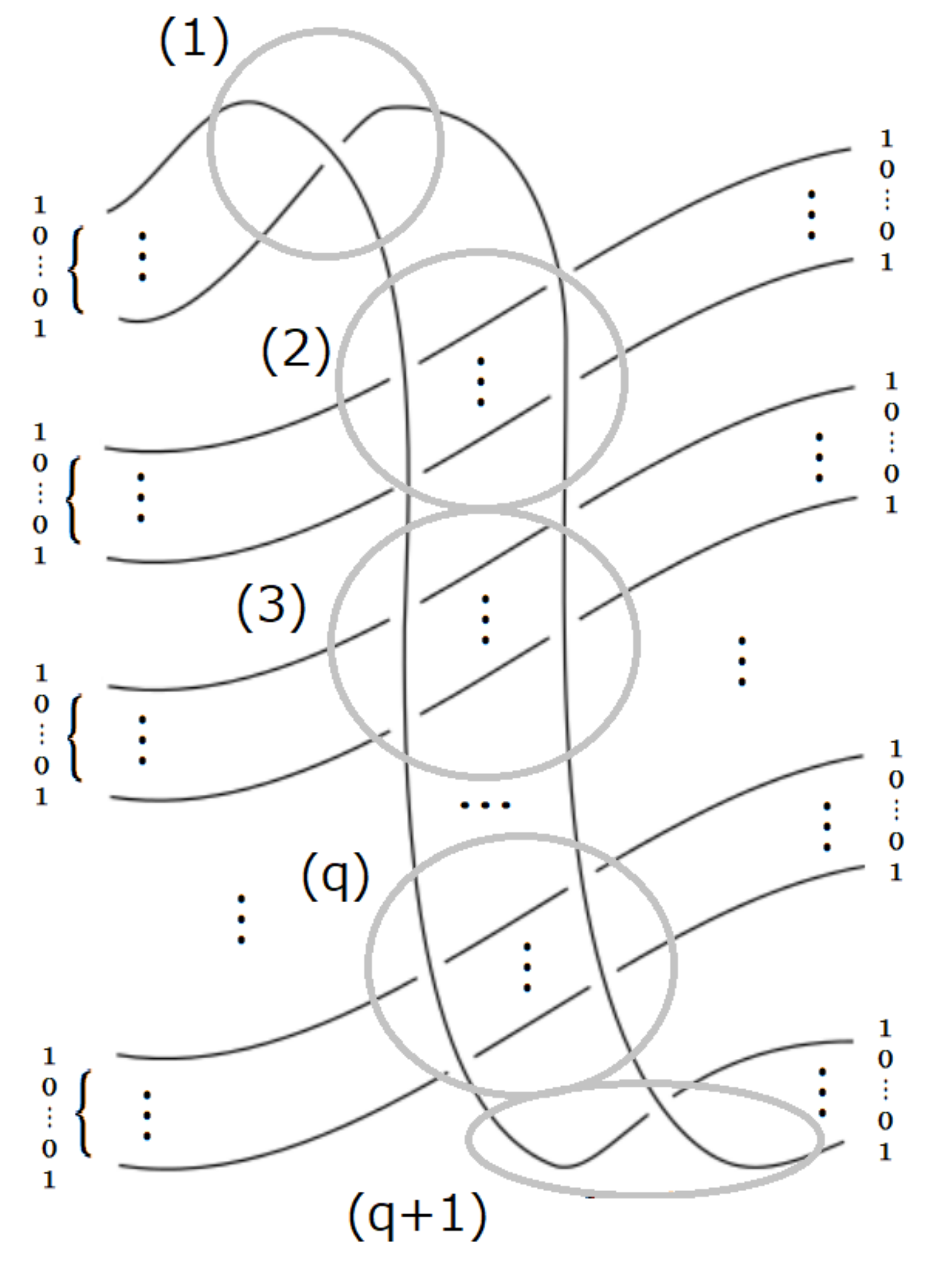}
\caption{}\label{torus9-1-2}
\end{center}
\end{figure}

\noindent Since $D$ is composed of $p$ copies of the local diagram in Figure \ref{torus9-1}~(left), it concludes that $D$ admits a $\mathbb{Z}$-coloring with only four colors $0$, $1$, $2$, and $3$. 

Since the torus link is non-splittable, $mincol_{\mathbb{Z}}(D)$ must be at least four by \cite[Theorem 3.1]{IchiharaMatsudo2}, and we conclude that $mincol_{\mathbb{Z}}(D) = 4$.
\end{proof}


\section{Five colors case}\label{sec:FiveColors}

In this section, we prove the following, which asserts the second half of the statement of Theorem~\ref{main-thm}. 

\begin{theorem}\label{q-even-thm}
Let $p, q,$ and $r$ be non-zero integers such that $p$ and $q$ are relatively prime, $|p|\geq q\geq 1$, and $r\geq 2$. 
Let $D$ be the standard diagram of $T(pr,qr)$ illustrated by Figure~\ref{torus}. 
Suppose that $T(pr,qr)$ is $\mathbb{Z}$-colorable, or, equivalently, $pr$ or $qr$ is even. 
Then $mincol_\mathbb{Z}(D)=5$ if $r$ is odd. 
\end{theorem}

To prove this, we recall the following lemma, which was obtained in \cite{IchiharaMatsudo2}.

\begin{lem}\label{lem1}
For a $\mathbb{Z}$-coloring $\gamma$ with $0=\min {\rm Im}\,\gamma$,
if an over arc at a crossing is colored by $0$, then the under arcs at the crossing are colored by $0$.
For a $\mathbb{Z}$-coloring $\gamma$ with $M=\max {\rm Im}\,\gamma$,
if an over arc at a crossing is colored by $M$, 
then the under arcs at the crossing are colored by $M$.
\end{lem}

\begin{proof}[Proof of Theorem~\ref{q-even-thm}]
We first show that the minimal coloring number $mincol_{\mathbb{Z}}(D)$ of $D$ is at most five, i.e., $mincol_{\mathbb{Z}}(D) \le 5$. 

In the following, we will find a $\mathbb{Z}$-coloring $\gamma$ with five colors on $D$ by assigning colors on the arcs of $D$. 
In the same way as in the proof of Theorem~\ref{thm_r:even}, we find a local $\mathbb{Z}$-coloring $\gamma$ on the local diagram shown in Figure \ref{torus9-1}~(left), and extend it to whole the diagram. 
Note that if $r$ is odd and $T(pr,qr)$ is $\mathbb{Z}$-colorable, then either $p$ or $q$ must be even.

First, suppose that $p$ is even. 
In this case, we start with setting $(\gamma(x_{i,1}),\gamma(x_{i,2}),\dots ,\gamma(x_{i,r})) = (1, 0, \dots, 0, 1)$ for any $i$. 
See Figure \ref{torus9-1}~(right) again. 
As illustrated in Figure \ref{torus9-6-2}~(left) and (right), we can extend $\gamma$ on the arcs in the regions $(1)$ and $(q+1)$ in Figure \ref{torus9-1}~(right). 

\begin{figure}[H]
\begin{center}
\includegraphics[width=.45\textwidth]{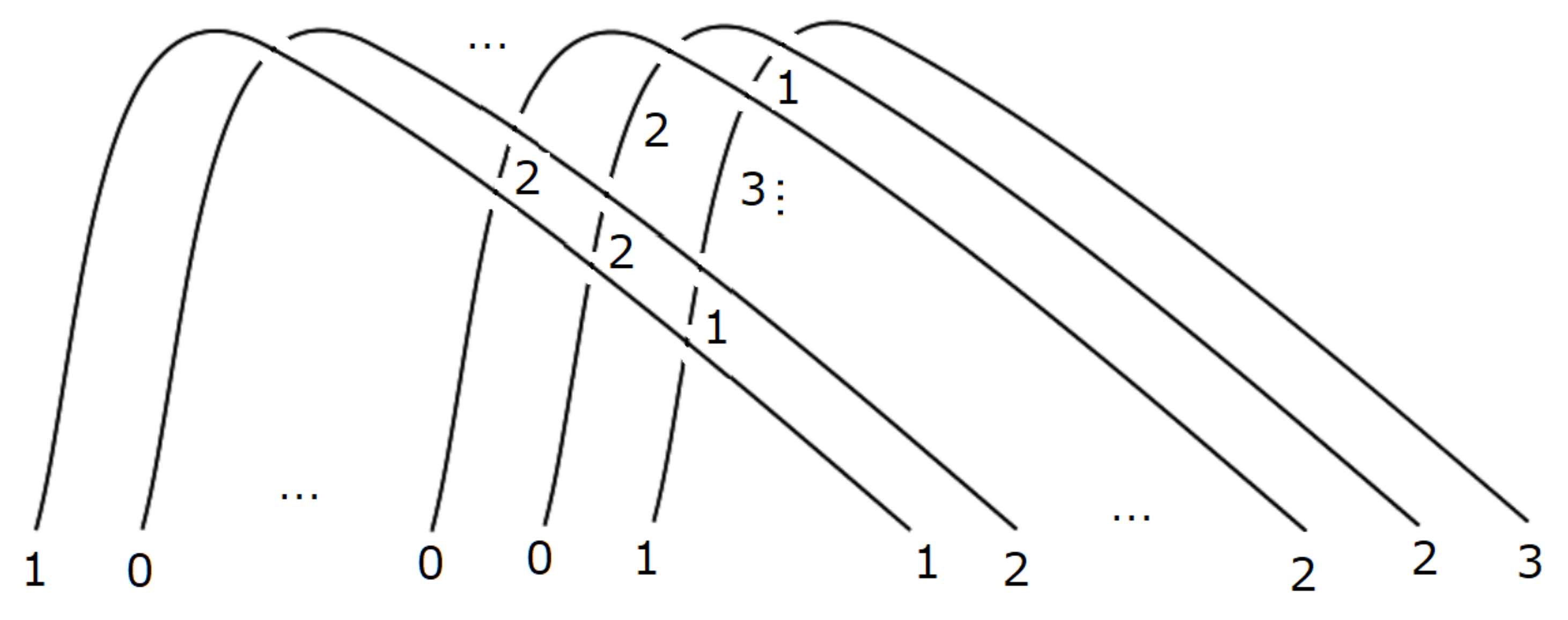}
\qquad 
\includegraphics[width=.45\textwidth]{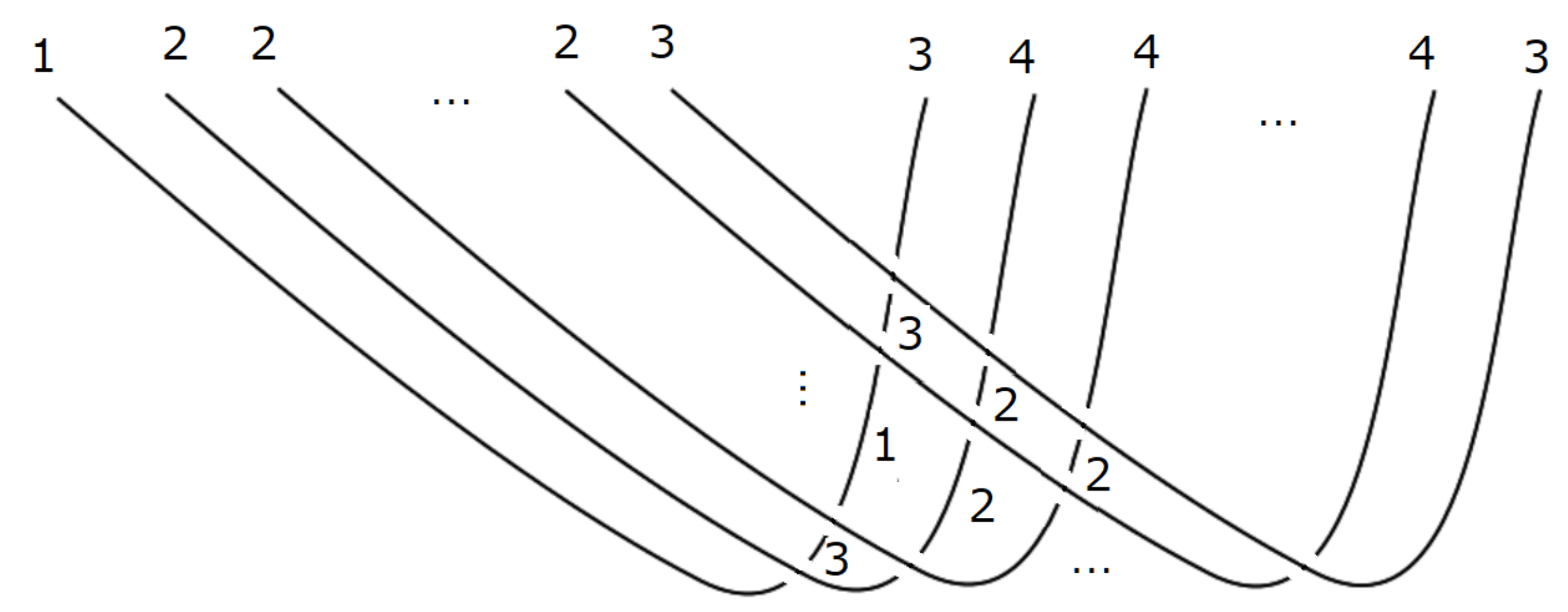}
\caption{}\label{torus9-6-2}
\end{center}
\end{figure}

\noindent 
As illustrated in Figure~\ref{torus9-7}~(left), we can extend $\gamma$ on the arcs in the regions $(2), (3), \dots, (q)$ in Figure \ref{torus9-1}~(right). 
Then, as shown in Figure \ref{torus9-7}~(right), we can extend $\gamma$ to all the arcs in Figure \ref{torus9-1}~(right). 

\begin{figure}[H]
\begin{center}
\includegraphics[width=.45\textwidth]{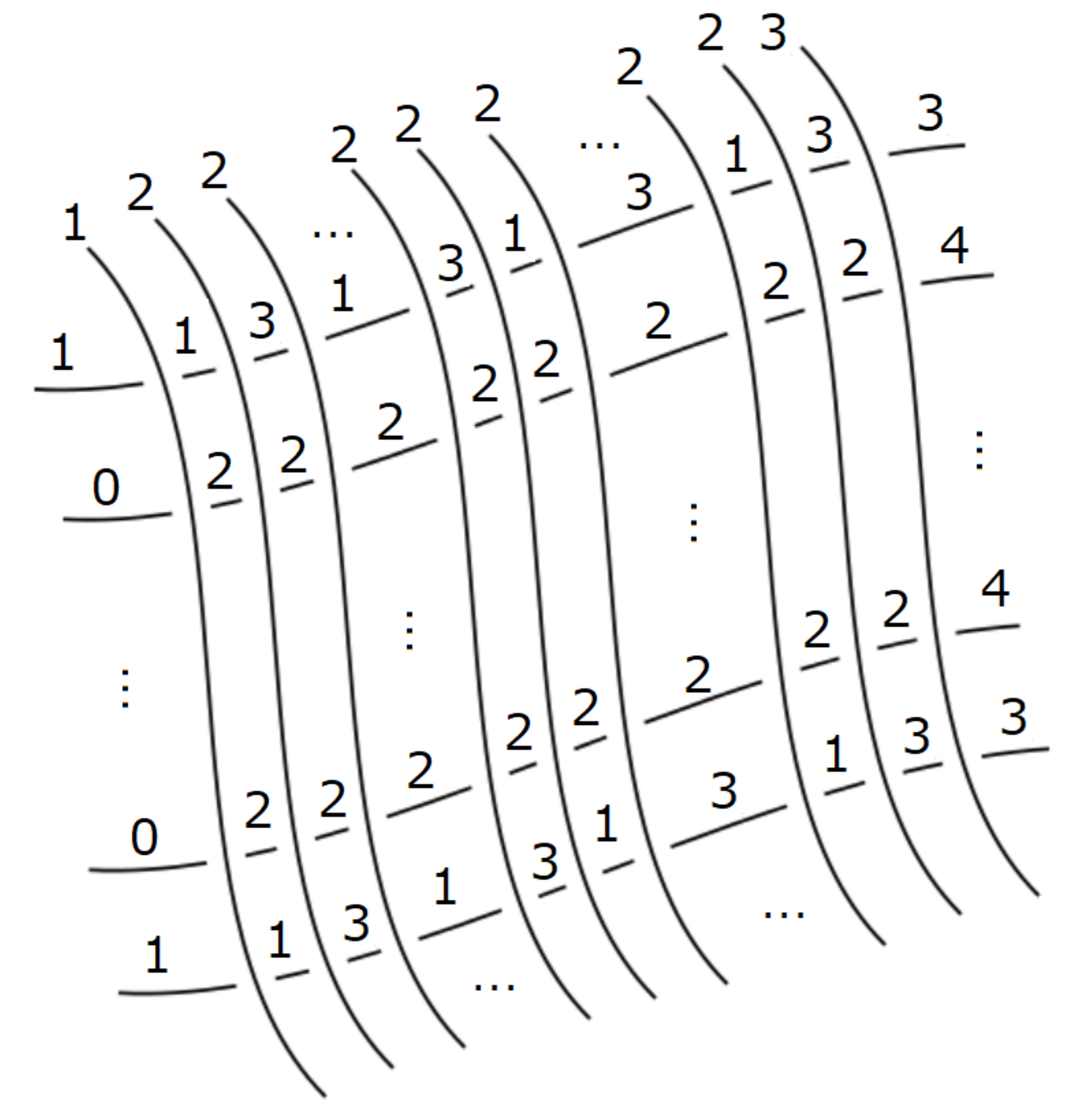}
\quad
\includegraphics[width=.45\textwidth]{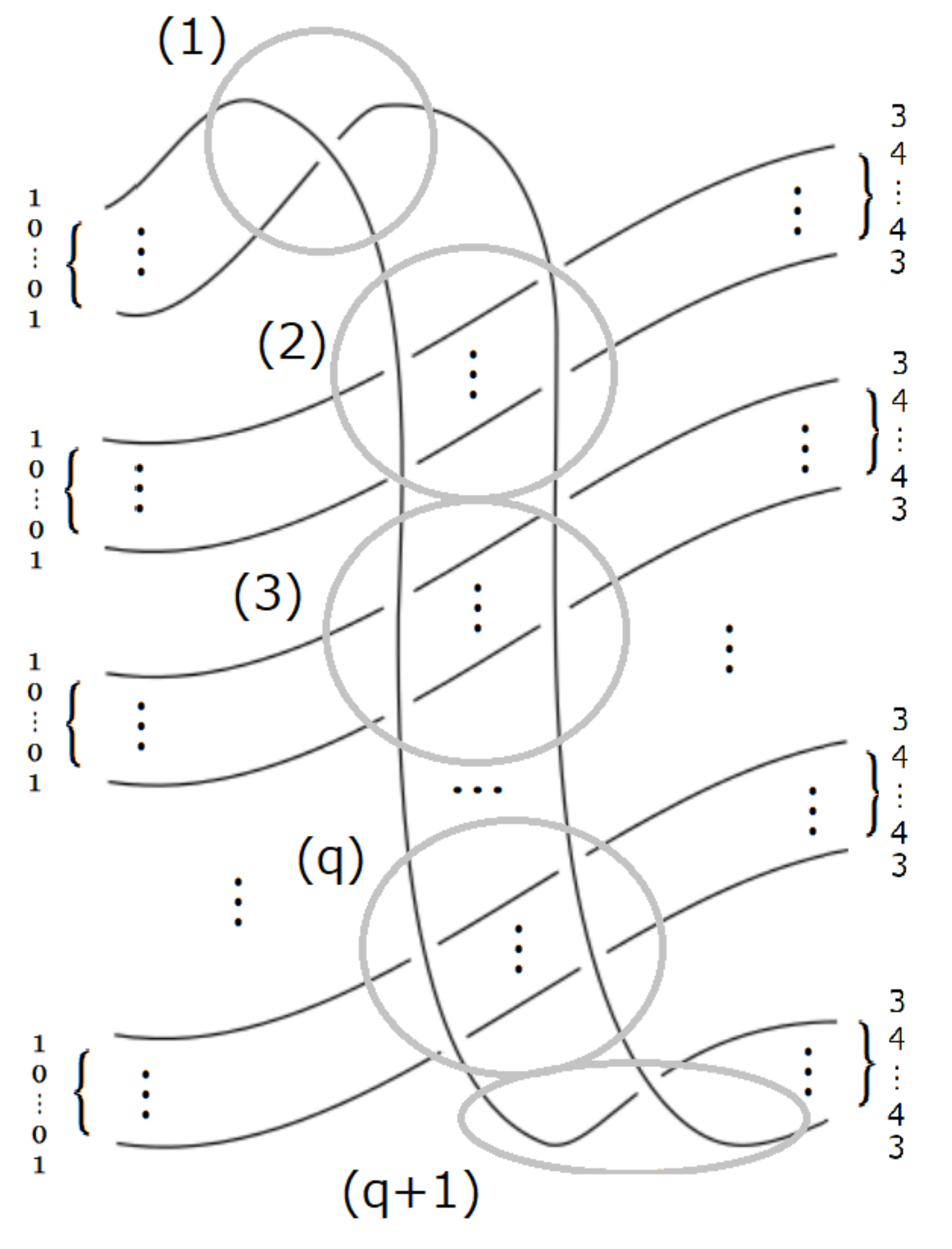}
\caption{}\label{torus9-7}
\end{center}
\end{figure}

\noindent Without contradicting to the condition of the coloring, we can connect  the local diagram in Figure~\ref{torus9-7}~(right) with the image of $\pi$-rotation. 
See Figure \ref{torus9-15}.

\begin{figure}[H]
\begin{center}
\includegraphics[width=.9\textwidth]{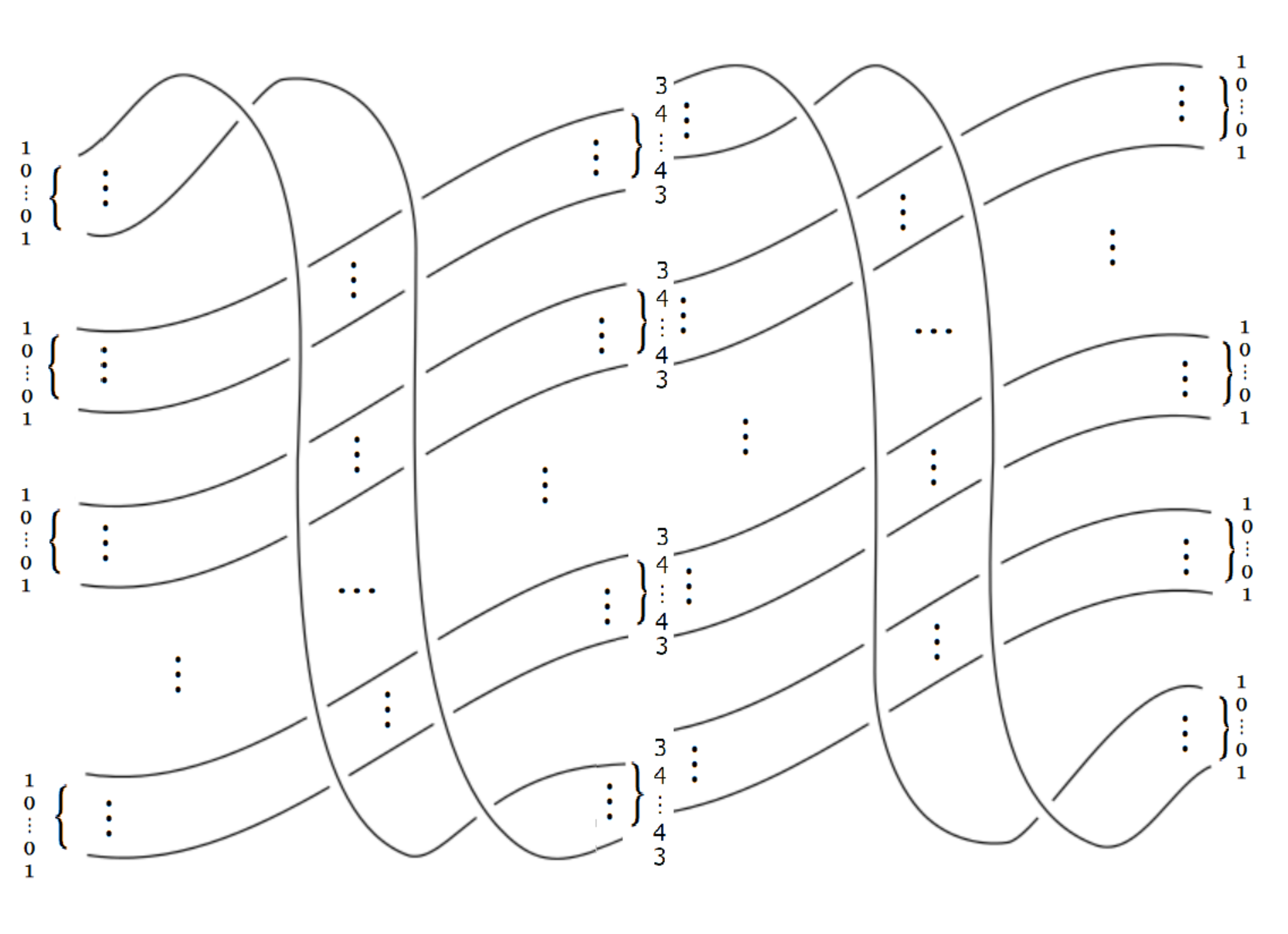}
\caption{}\label{torus9-15}
\end{center}
\end{figure}

\noindent That is, since $p$ is assumed to be even, $D$ admits a $\mathbb{Z}$-coloring which is composed by connecting $p/2$ local diagrams illustrated by Figure \ref{torus9-15}.
It concludes that the colors of this $\mathbb{Z}$-coloring are $\{0, 1, 2, 3, 4\}$, that is, the $\mathbb{Z}$-coloring is represented by five colors.

Next, suppose that $r$ is odd and $q$ is even. 
In this case, we start with setting 
$$
(\gamma(x_{i,1}),\gamma(x_{i,2}),\dots ,\gamma(x_{i,r})) =
\begin{cases}
(2, 1,  \dots, 1, 0) & \text{ if $i$ is odd,} \\
(0, 1,  \dots, 1, 2) & \text{ if $i$ is even.}
\end{cases}
$$
See Figure~\ref{torus9-8}~(left). 

\begin{figure}[H]
\centering
$$\begin{tabular}{c}\includegraphics[width=.4\textwidth]{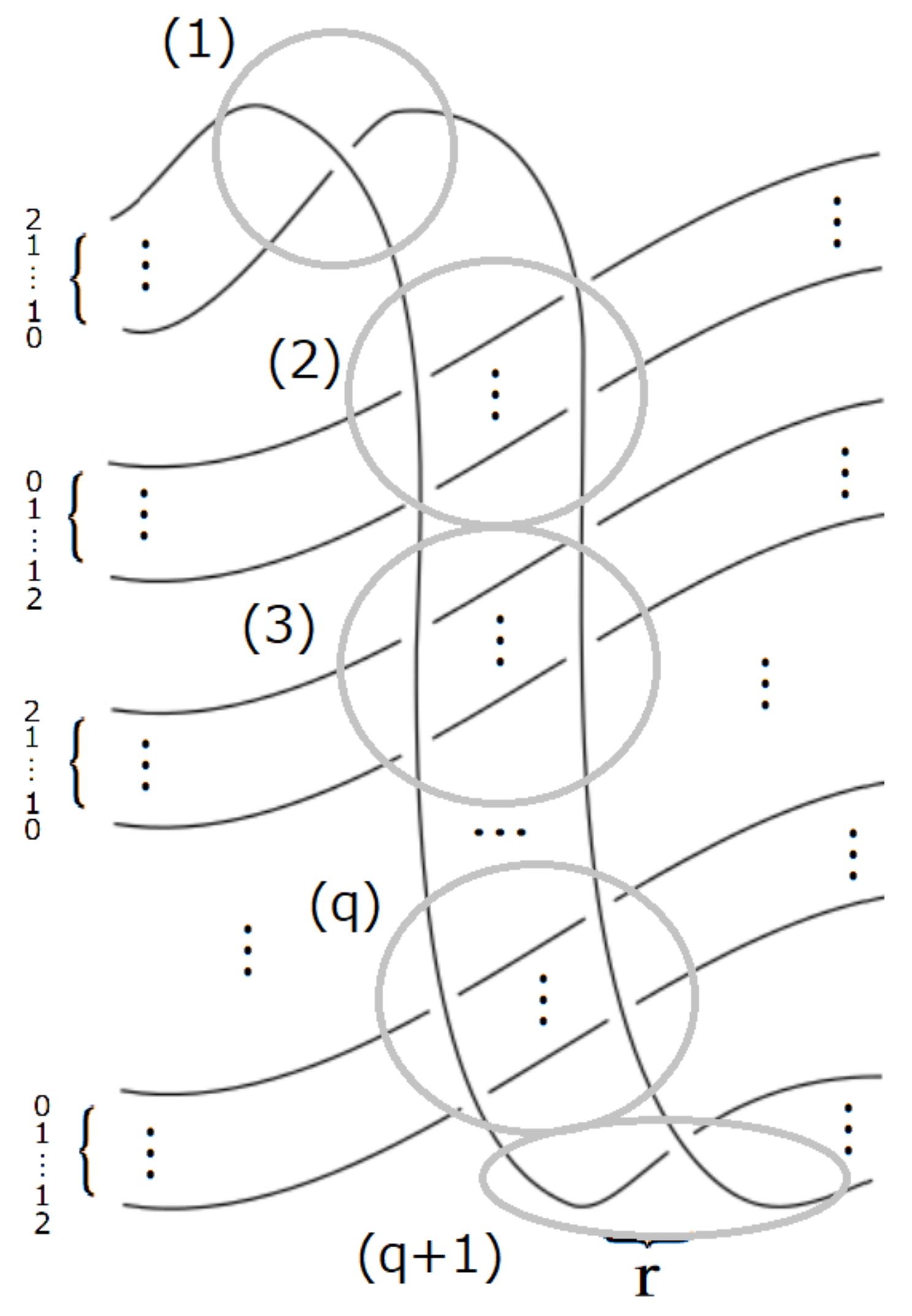}\end{tabular}
\quad
\begin{tabular}{c}\includegraphics[width=.5\textwidth]{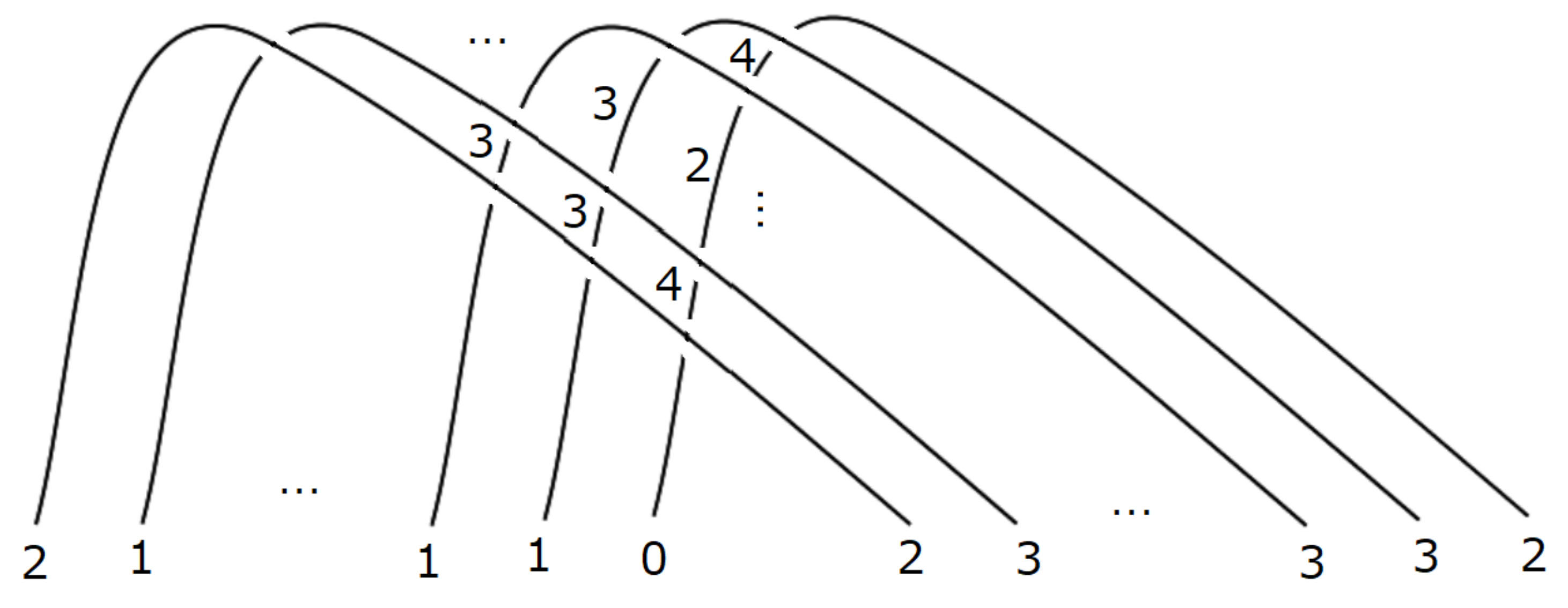}\end{tabular}$$
\caption{}\label{torus9-8}
\end{figure}

\noindent As illustrated in Figure \ref{torus9-8}~(right), we can extend $\gamma$ on the arcs in the regions $(1)$ and $(q+1)$ in Figure \ref{torus9-1}~(right). 
And, as illustrated in Figure \ref{torus9-10}, 
we can extend $\gamma$ on the arcs in the regions $(2), (3), \dots, (q)$ in Figure \ref{torus9-1}. 

\begin{figure}[H]
\begin{center}
\includegraphics[width=\textwidth]{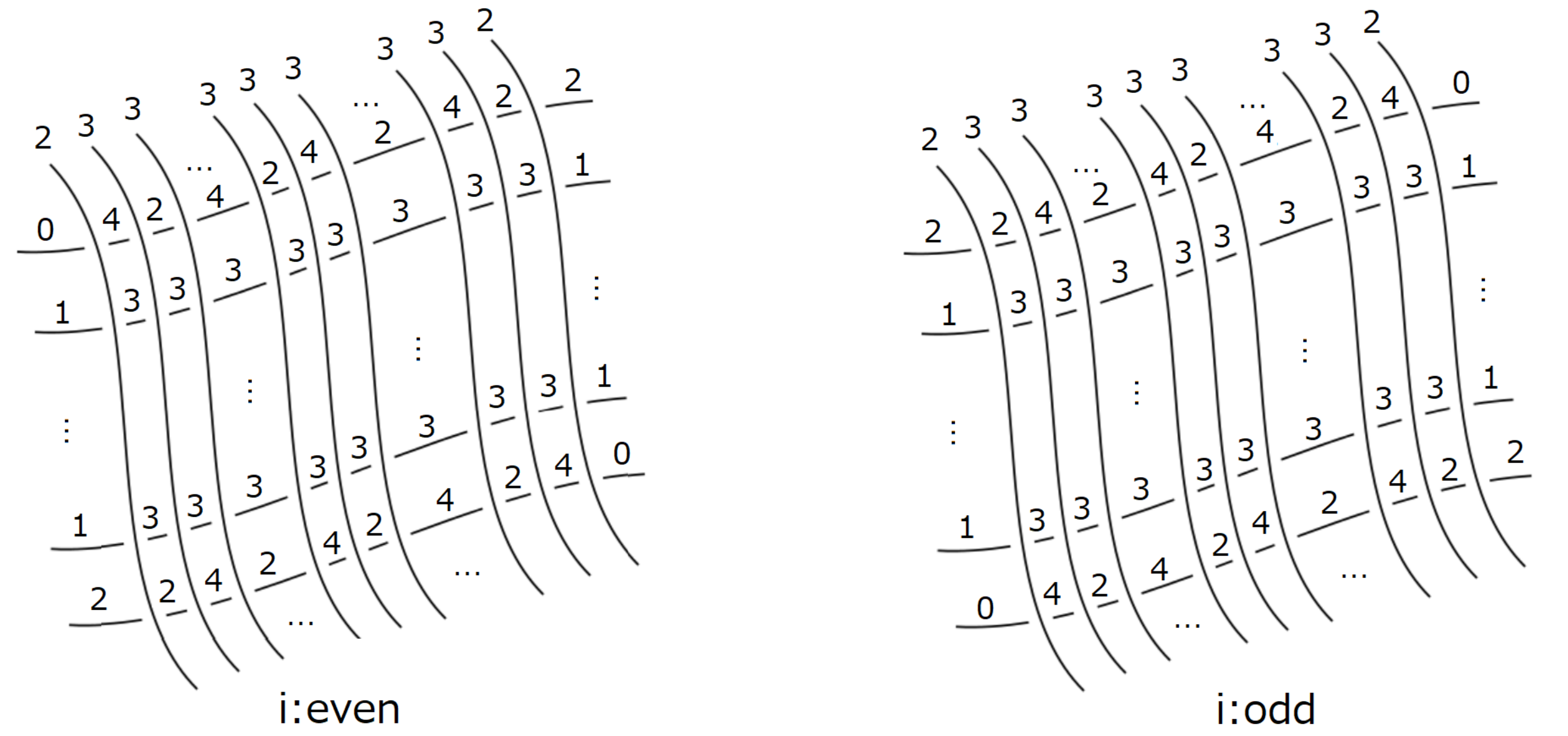}
\caption{}\label{torus9-10}
\end{center}
\end{figure}

\noindent Now, $\gamma$ can be extended on all the arcs in Figure \ref{torus9-8-2}.

\begin{figure}[H]
\begin{center}
\includegraphics[width=.5\textwidth]{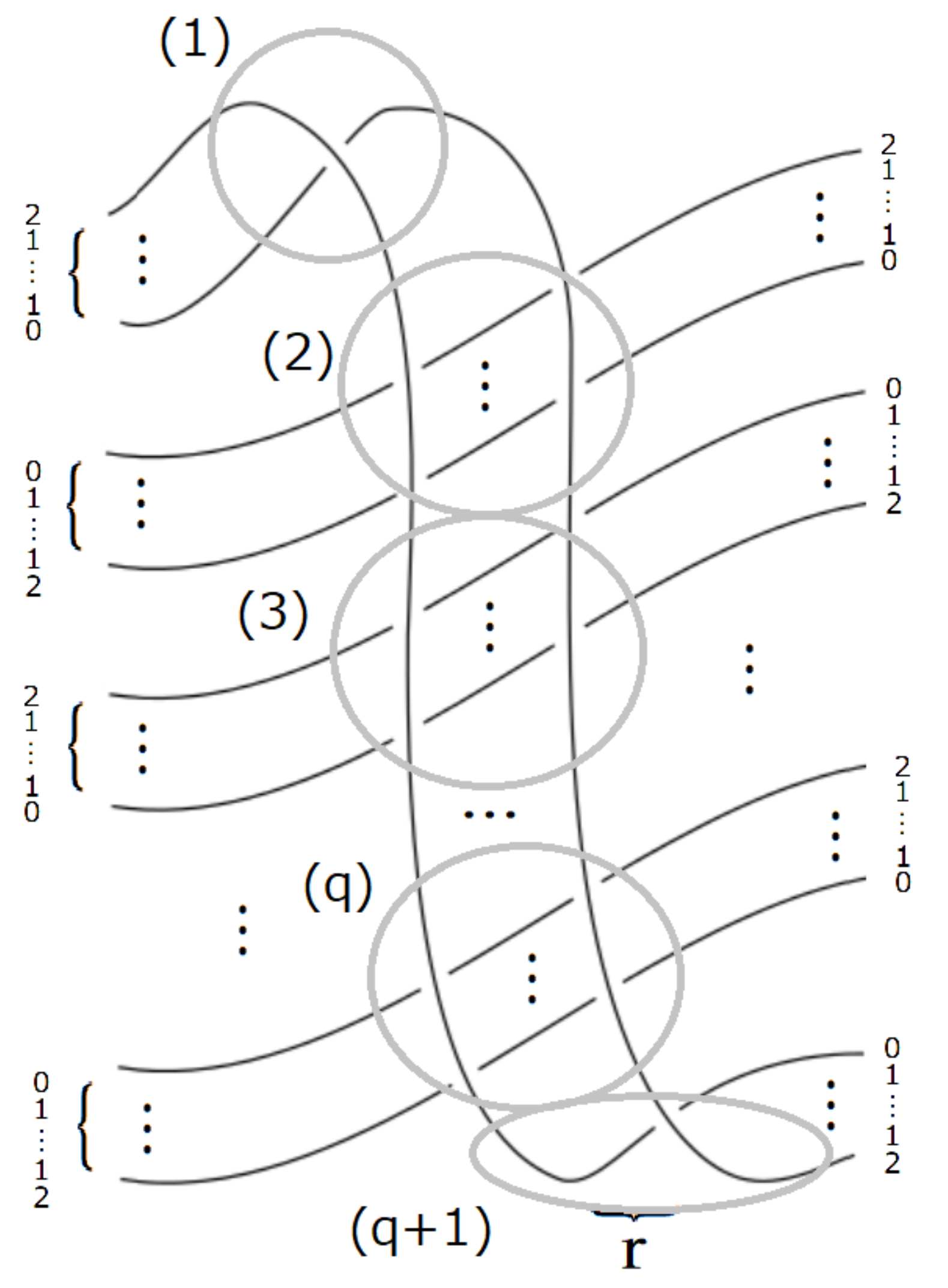}
\caption{}\label{torus9-8-2}
\end{center}
\end{figure}

\noindent Since $D$ is composed of $p$ copies of the local diagram in Figure \ref{torus9-1}, it concludes that $D$ admits a $\mathbb{Z}$-coloring with only five colors $0$, $1$, $2$, $3$, and $4$. 

Consequently, we obtain that $mincol_{\mathbb{Z}}(D) \le 5$. 

\par
We next show that the minimal coloring number $mincol_{\mathbb{Z}}(D)$ of $D$ is at least five, i.e.,  $mincol_{\mathbb{Z}}(D) \ge 5$. 

Suppose for a contradiction that the diagram $D$ as shown in Figure \ref{torus} admits 
a non-trivial $\mathbb{Z}$-coloring $\gamma$ with only four colors. 
By \cite[Theorem 3.2]{IchiharaMatsudo2}, we may assume that the image of $\gamma$ is $\{0, 1, 2, 3\}$. 
By Lemma~\ref{lem1}, we see that the over arcs are colored by $1$ or $2$, otherwise $\gamma$ have to be trivial. 
Thus there exist only three ways to color a crossing: 
\begin{figure}[H]
\centering
\includegraphics[width=7cm]{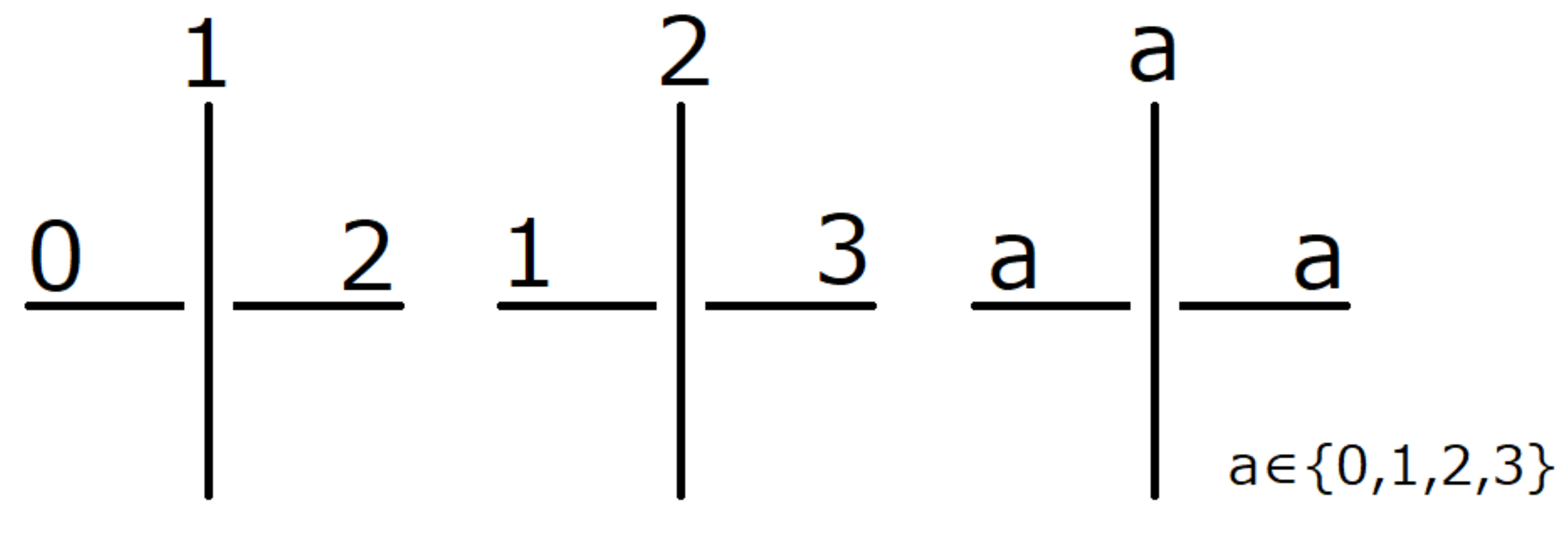}
\caption{Colors at a crossing}\label{crossings}
\end{figure}

\noindent Here we see that, in a component including an arc colored by 1 or 3 (resp. 0 or 2), the arcs are always colored by odd (resp. even) numbers, by the condition of the $\mathbb{Z}$-coloring. 

Then, since the number of the component is the odd number $r$, 
either the number of the components colored by odd numbers or that by even numbers is odd.
Since the linear function $X \mapsto -X + 3$ on $\mathbb{Z}$ switches these two cases, 
we only have to consider the former case; 
then the number of the over arcs colored by $1$ is odd 
in the $r$ parallel over arcs as in Figure \ref{torus10}.

\begin{figure}[H]
\begin{center}
\includegraphics[height=7cm,clip]{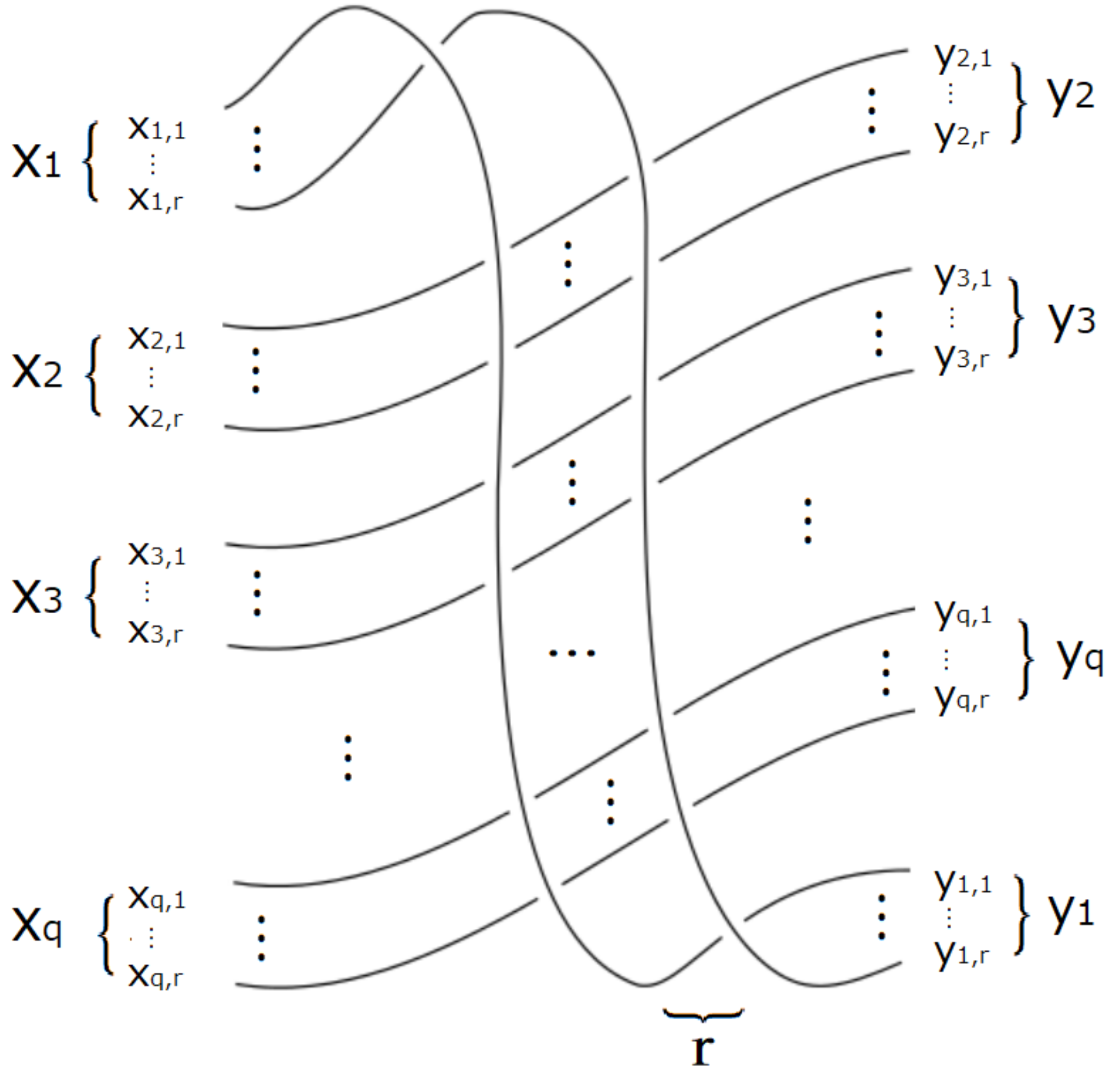}
\caption{}\label{torus10}
\end{center}
\end{figure}


In the case $q=1$, we consider the $r$ parallel arcs $\mathbf{x} =\{x_1, \dots ,x_r\}$ 
as shown in Figure \ref{torus11}. 

\begin{figure}[H]
\begin{center}
\includegraphics[height=3cm,clip]{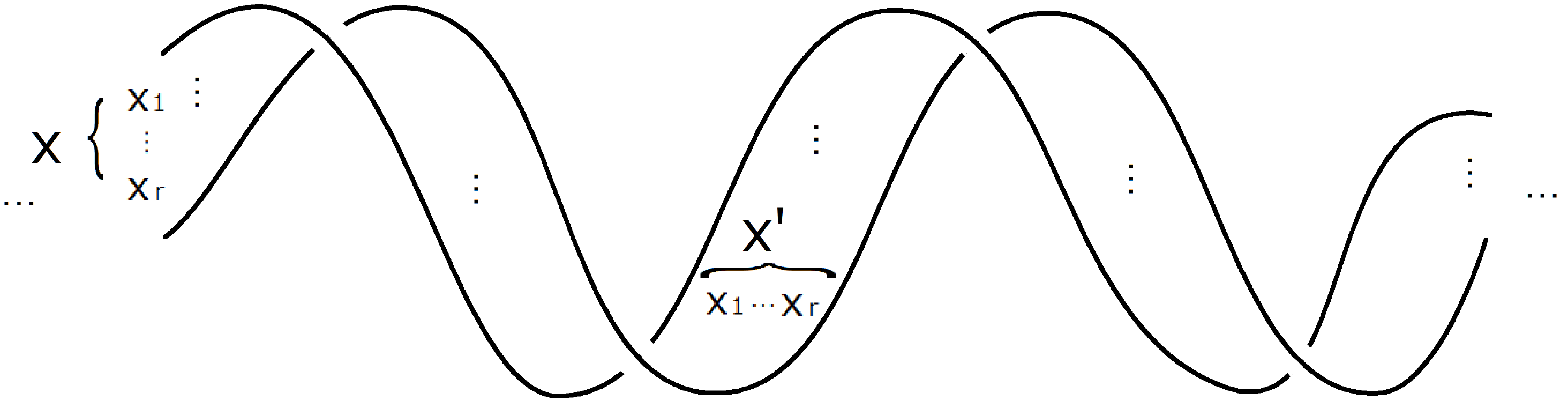}
\caption{}\label{torus11}
\end{center}
\end{figure}

\noindent The colors on $\mathbf{x}$ is represented by $\gamma(\mathbf{x})=(\gamma(x_1), \dots ,\gamma(x_r))$.  

By the assumption that $\gamma$ is a non-trivial $\mathbb{Z}$-coloring, the diagram has an over arc colored by $2$. 
Hence we can label $\mathbf{x}$ to have $x_r$ be colored by $2$. 
That is, we consider the case the colors of $\mathbf{x}$ is $\gamma(\mathbf{x})=(a_1, a_2,\dots,a_{r-1},2)$, 

The arc $x_r$ turns into $x'_r$ by passing through the odd arcs colored by $1$ and even arcs colored by $2$. 
Here, as the diagram has exactly four colors $0, 1, 2, 3$, the arc $x'_r$ is colored by $0$.
That is, there exists an over arc colored by $0$ as shown in Figure \ref{torus12}.

\begin{figure}[H]
\begin{center}
\includegraphics[height=3cm,clip]{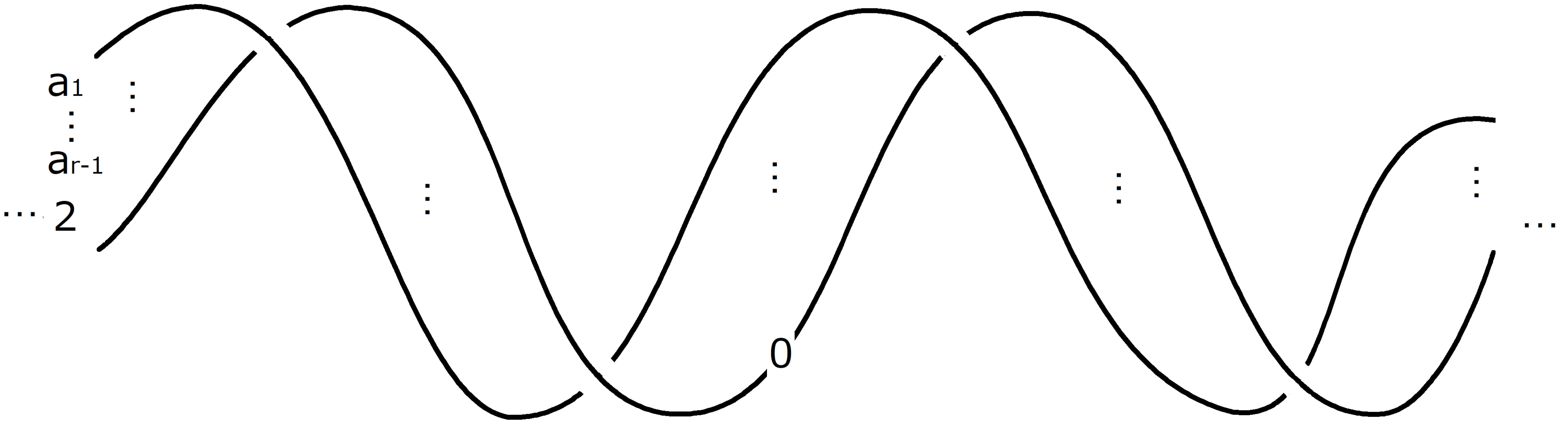}
\caption{}\label{torus12}
\end{center}
\end{figure}

\noindent This is a contradiction; the colors on over arcs are $1$ or $2$.


In the case $q\ge 2$, 
We divide $qr$ parallel arcs into $q$ subfamilies $\mathbf{x}_1, \dots, \mathbf{x}_q$ 
and $\mathbf{y}_1, \dots, \mathbf{y}_q$ as depicted in Figure \ref{torus10}. 
The colors on $\mathbf{x}_i=(x_{i,1}, \dots ,x_{i,r})$ are represented by $\gamma(\mathbf{x}_i)=(\gamma(x_{i,1}), \dots ,\gamma(x_{i,r}))$.  

By the condition of a $\mathbb{Z}$-coloring, the colors on $\mathbf{y}_i$ are expressed by using a linear function $f$
as $\gamma(\mathbf{y}_i) = \{f(\gamma(x_{i,1})), \dots ,f(\gamma(x_{i,r}))\}$ with $i=2,3,\dots,q$.
Then we see $f(0)=2$ and $f(2)=0$, as the number of the over arcs colored by $1$ is odd in the $r$ over arcs in Figure \ref{torus10}.

\begin{figure}[H]
\begin{center}
\includegraphics[height=4cm,clip]{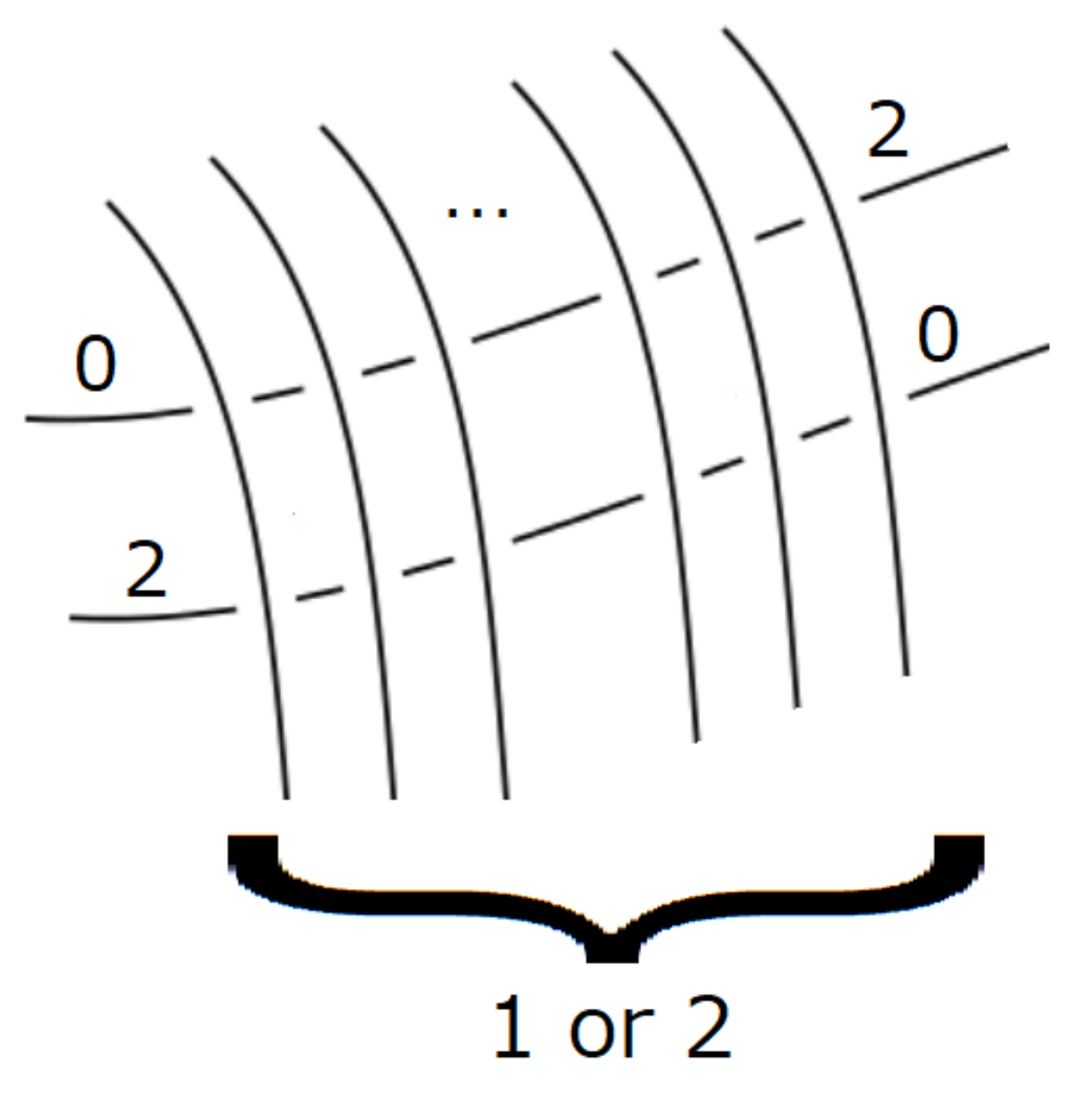}
\caption{}\label{torus14}
\end{center}
\end{figure}

Here we obtain $f(X)=-X+2$ and see $f(3)=-1$. 
That is, if there exists an arc colored by $3$ in any parallel arcs $\mathbf{x}_i$ ($i=2,\dots,q$), there exists an arc colored by $-1$ in the parallel arcs $\mathbf{y}_{i}$. 
It is a contradiction since the image of $\gamma$ is $\{0,1,2,3\}$.

By the assumption, the diagram has an arc colored by $3$.
If there exist no arcs colored by $3$ in parallel arcs $\mathbf{x}_i$, 
we relabel the arcs to have the arc colored by $3$ be one of parallel arcs. 
Furthermore, if there exist an arc colored by $3$ in $\mathbf{x}_1$, 
we see the outside of the local diagram as shown in Figure \ref{torus10}  
and we relabel to have $\mathbf{x}_1$ as $\mathbf{y}_1$. 
Since $f$ is an involution, this completes the theorem.

\end{proof}


\section{$\mathbb{Z}$-colorings of torus link diagrams}\label{sec:description}

In this section, we give complete classifications of all $\mathbb{Z}$-colorings of the standard diagram of $T(a,b)$ (Proposition~\ref{col-prop}) and of all $\mathbb{Z}$-colorings by only four colors of $T(a,b)$ (Proposition~\ref{4col-prop}). 

To achieve these, we prepare a theorem on $\mathbb{Z}$-colorings of $n$-parallels of knots. 
For a knot diagram $D$, we obtain another diagram $D^{(n)}$ by replacing the string with $n$ parallel copies of it; we call $D^{(n)}$ the $n$-\textit{parallel} of $D$. 
Remark that $D^{(n)}$ represents the $(nw,n)$-cable link of the knot represented by $D$, where $w$ is the writhe of $D$. 
In the following, we denote the set of $\mathbb{Z}$-colorings of a link diagram $D$ by ${\rm Col}_{\mathbb{Z}}(D)$.

\begin{theorem}\label{parallel-thm}
Let $D$ be an oriented knot diagram and $D^{(n)}$ the $n$-parallel of $D$. We fix any arc of $D$ and let $\gamma_1, \dots, \gamma_n$ be the corresponding $n$ arcs of $D^{(n)}$. We define a homomorphism $r: {\rm Col}_{\mathbb{Z}}(D^{(n)}) \to \mathbb{Z}^n$ as
\[ r(\mathcal{C}) = (\mathcal{C}(\gamma_1), \dots, \mathcal{C}(\gamma_n)). \]
Then, $r$ is injective and the image ${\rm Im}\,r$ of $r$ is equal to
\[ \left\{ \begin{array}{ll}
\{(a_1, \dots, a_n) \mid w(a_1 - a_2 + a_3 - \cdots - a_n) = 0\} & \text{if $n$ is even,}\\[2pt]
\{(a, \dots, a) \mid a \in \mathbb{Z}\} & \text{if $n$ is odd and $w$ is odd,}\\[2pt]
\mathbb{Z}^n & \text{if $n$ is odd and $w$ is even,}
\end{array} \right. \]
where $w$ is the writhe of $D$.
\end{theorem}

\begin{remark}\label{parallel-rem}
Theorem \ref{parallel-thm} states that the colors $a_1,\dots,a_n$ of the $n$-paralleled arcs of an arc $\gamma$ determine the whole coloring, especially the colors $a'_1,\dots,a'_n$ of the $n$-parallel of another arc $\gamma'$. 
As we see in the proof, we can calculate $a'_1,\dots,a'_n$ from a cyclic-rack coloring of $D$. 
For example, if $n$ is even and $w \neq 0$, we always have $a'_i = a_i$ for any $i$. 
If $n$ is odd and $w$ is even, we track the string of $D$ from $\gamma$ to $\gamma'$ and count the number of times of passing under arcs; if it is even, then $a'_i = a_i$, and otherwise, $a'_i = -a_i + 2(a_1 - a_2 + \cdots + a_n)$.
\end{remark}

We include a proof of Theorem~\ref{parallel-thm} in Appendix, for it is rather independent from the other contents of the paper. 

\subsection{Determining $\mathbb{Z}$-colorings}
Let $B(a,b)$ denote the braid (diagram) illustrated in Figure \ref{braid-fig}. 
We can regard the torus link $T(a,b)$ as the closure of $B(a,b)$. 
Then, an assignment of colors to the $b$ left ends of $B(a,b)$ determines a $\mathbb{Z}$-coloring of $B(a,b)$, and if the resulting colors of the right ends coincide with the left colors, it gives a $\mathbb{Z}$-coloring of the standard diagram of $T(a,b)$.
\begin{figure}[htb]
\centering
\includegraphics[width=7cm]{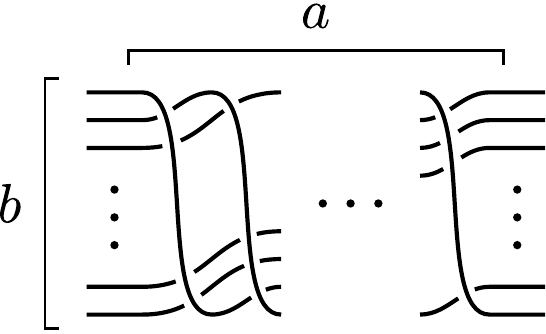}
\caption{$B(a,b)$}\label{braid-fig}
\end{figure}
\par Let $p,q$, and $r$ be nonzero integers such that $|p| \geq q \geq 1$ and $r \geq 2$, and assume that $p$ and $q$ are mutually prime. 
We divide the $qr$ arcs of the left end of $B(pr,qr)$ into $q$ subfamilies ${\bm x}_1, \dots, {\bm x}_q$ as depicted in Figure \ref{torus9-1}~(left). 
Let $A$ be the set of the $qr$-tuples of integers which give a $\mathbb{Z}$-coloring of the standard diagram $D$ of $T(pr,qr)$, i.e.,
$$A = \left\{({\bm a}_1, \dots, {\bm a}_q) \in (\mathbb{Z}^r)^q \middle|
\begin{tabular}{l} the assignment of ${\bm a}_1, \dots, {\bm a}_q \in \mathbb{Z}^r$ to ${\bm x}_1, \dots, {\bm x}_q$\\
defines a $\mathbb{Z}$-coloring of $D$\end{tabular}\right\}.$$
\par The following proposition describes the coloring of torus links:

\begin{prop}\label{col-prop}
We have
$$A = \left\{ \begin{tabular}{ll} $\{({\bm a}, \dots, {\bm a}) \mid {\bm a} \in \mathbb{Z}^r, \Delta({\bm a}) = 0\}$ & if $r$ is even,\\
$\{({\bm a}, \dots, {\bm a}) \mid {\bm a} \in \mathbb{Z}^r\}$ & if $r$ is odd and $p$ is even,\\
$\{({\bm a}, \tau({\bm a}), {\bm a}, \dots, \tau({\bm a})) \mid {\bm a} \in \mathbb{Z}^r\}$ & if $r$ is odd and $q$ is even,\\
\end{tabular} \right.$$
where $\Delta({\bm a}) = a_1- a_2 + \cdots +(-1)^r a_r \in \mathbb{Z}$ and $\tau({\bm a}) = (-a_i + 2\Delta({\bm a}))_i \in \mathbb{Z}^r$ for ${\bm a} = (a_1,\dots,a_r) \in \mathbb{Z}^r.$
\end{prop}
\begin{proof}
Let $B_0$ be the tangle diagram depicted in Figure~\ref{b0-fig} and $B_0^{(r)}$ the $r$-parallel of $B_0$. 
We denote the closures of $B_0$ and $B_0^{(r)}$ by $D_0$ and $D_0^{(r)}$, respectively. 
Since $D_0$ represents the knot $T(p,q)$, we can apply Theorem \ref{parallel-thm} to $D_0$ and its $r$-parallel $D_0^{(r)}$ to determine the colorings of $D_0^{(r)}$. 
Furthermore, we should remark that $B_0^{(r)}$ and $B(pr,qr)$ are isotopic; we find the colorings of $D$ from those of $D_0^{(r)}$, using an isotopy which takes $D_0^{(r)}$ to $D$ and fixes the arcs $x_{i,j}$ ($1 \leq i \leq q, 1 \leq j \leq r$) and their colors.

\begin{figure}[h]
\centering
\includegraphics[width=7cm]{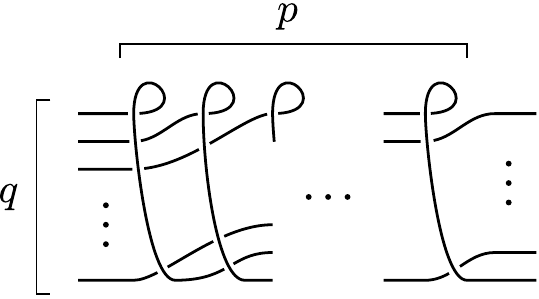}
\caption{$B_0$}\label{b0-fig}
\end{figure}

\par By Theorem \ref{parallel-thm} (and Remark \ref{parallel-rem}), a coloring of $D_0^{(r)}$ is determined by the color ${\bm a} \in \mathbb{Z}^r$ of $r$ arcs ${\bm x}_1$, and the whole coloring is found from a cyclic-rack coloring of $D_0$. 
If $r$ is even, the color ${\bm a}$ has to satisfy the condition $\Delta({\bm a}) = 0$ and then the color of ${\bm x}_i$ is equal to ${\bm a}$ for any $i$. 
In the other cases, we can choose any ${\bm a} \in \mathbb{Z}^r$, and we consider the cyclic rack $C_2$ of order $2$ to examine the whole coloring. A brief calculation finds a $C_2$-coloring of $D_0$, which colors the left $q$ arcs
$$\left\{ \begin{array}{ll} 0,0,0,0,\dots,0 & \text{if $p$ is even and $q$ is odd,}\\
0,1,0,1,\dots,0 & \text{if $p$ is odd and $q$ is even.} \end{array} \right.$$
In terms of Remark \ref{parallel-rem}, a color $0 \in C_2$ corresponds to ``even'' and $1 \in C_1$ to ``odd''; as in the remark, we associate ${\bm a}$ to the $r$-paralleled arcs of an arc if it has a color $0$, and associate $\tau({\bm a})$ otherwise. This completes the proposition.
\end{proof}


\subsection{Coloring with four colors}
Let $p, q,$ and $r$ be non-zero integers such that $p$ and $q$ are relatively prime, $|p|\geq q\geq 1$, and $r\geq 2$. 
By Theorem \ref{main-thm}, the standard diagram $D$ of $T(pr,qr)$ admits a $\mathbb{Z}$-coloring $\gamma$ with four colors if and only if $r$ is even. 
In this case, it is sufficient to consider the case where ${\rm Im}\,\gamma = \{0,1,2,3\}$ by \cite[Theorem 3.2]{IchiharaMatsudo2}. Let $A^{(4)}$ be the set of the $qr$-tuples which give such $\mathbb{Z}$-colorings:
$$A^{(4)} = \left\{({\bm a}_1, \dots, {\bm a}_q) \in (\mathbb{Z}^r)^q \middle| 
\begin{tabular}{l} the assignment of ${\bm a}_1, \dots, {\bm a}_q \in \mathbb{Z}^r$ to ${\bm x}_1, \dots, {\bm x}_q$\\
defines a $\mathbb{Z}$-coloring of $D$ with the four colors $\{0,1,2,3\}$\end{tabular}\right\}.$$
Here, we regard $D$ as the closure of $B(pr,qr)$ and denote the subfamilies of the left $qr$ arcs by ${\bm x}_1,\dots,{\bm x}_q$ as shown in Figure \ref{torus9-1} (left).

\begin{prop}\label{4col-prop}
We have
$$A^{(4)} = \left\{({\bm a},\dots,{\bm a}) \,\middle|\, {\bm a} \in A^{(4)}_{01} \cup A^{(4)}_{12} \cup A^{(4)}_{23} \right\}\backslash \{(1,\dots,1),(2,\dots,2)\},$$
where
\begin{align*}
A^{(4)}_{01} &= \left\{(a_1, \dots, a_r) \in \{0,1\}^r \,\middle|\, a_1 = a_r = 1, a_{2i} = a_{2i+1} \;(i=1,\dots,r/2-1)\right\},\\
A^{(4)}_{12} &= \left\{(a_1, \dots, a_r) \in \{1,2\}^r \,\middle|\, a_{2i-1} = a_{2i} \;(i=1,\dots,r/2)\right\},\\
A^{(4)}_{23} &= \left\{(a_1, \dots, a_r) \in \{2,3\}^r \,\middle|\, a_1 = a_r = 2, a_{2i} = a_{2i+1} \;(i=1,\dots,r/2-1)\right\}.
\end{align*}
\end{prop}
\begin{proof}
By Proposition \ref{col-prop}, a member of $A^{(4)}$ is expressed as $({\bm a}, \dots, {\bm a})$, where ${\bm a} = (a_1,\dots, a_r)$ is an $r$-tuple of integers with $\Delta({\bm a}) = a_1 - a_2 + \dots - a_r$ equal to $0$. 
We have to show that ${\bm a} \in A^{(4)}_{01} \cup A^{(4)}_{12} \cup A^{(4)}_{23}$.
\par Let $S$ be the set $\{a_1,\dots,a_r\}$. We first assert that $S = \{0,1\},\{1,2\},$ or $\{2,3\}$. 
To show $S \not\supset \{0,2\}$, we suppose that $\{a_i,a_j\} = \{0,2\} \;(i < j)$. 
The arc $x_{1,i}$ goes under the $i-1$ arcs to become an over arc $x'_i$ with color $a'_i$, whereas $x_{1,j}$ goes under the same arcs to become $x'_j$ with color $a'_j$ and then passes under $x'_i$. 
Here, we should remark that $|a'_i - a'_j| = 2$ since $x_{1,i}$ and $x_{1,j}$ goes under the same arcs, but such a crossing is not allowed, as seen in the proof of Theorem \ref{q-even-thm}. 
Thus we have $S \not\supset \{0,2\}$, and in the same way, we can check that $S \not\supset \{1,3\}, \{0,3\}$. 
Therefore $S$ is a set of two consecutive integers, as asserted.
\par In the case where $S = \{1,2\}$, we see that $a_1 = a_2$; otherwise, the color $-a_2 + 2a_1$ of the over arc next to $x_{1,1}$ would not be $1$ or $2$. 
After two twists, the colors of the arcs shift cyclically to be $(a_3, a_4, \dots, a_2)$ since passing under two strings with a same color does not change the color. 
Then, the argument above shows that $a_3 = a_4$. 
Repeating this, we find $a_{2i - 1} = a_{2i}$ for $i = 1,\dots,r/2$, i.e., ${\bm a} \in A^{(4)}_{12}$.
\par In the case where $S = \{0,1\}$ (resp. $\{2,3\}$), the colors $a_1$ and $a_r$ of over arcs have to be $1$ (resp. $2$). 
After $x_{1,1}$ comes to the bottom, the colors of the arcs are $(-a_2 + 2a_1, -a_3 + 2a_1,\dots, -a_r + 2a_1, a_1)$, and this is also a member of $A^{(4)}$. 
Since the set $\{a_1, -a_2 + 2a_1,,\dots, -a_r + 2a_1\}$ is equal to $\{1,2\}$, we have $a_{2i} = a_{2i+1} \;(i=1,\dots,r/2-1)$ as shown above. 
This means that ${\bm a} \in A^{(4)}_{01}$ (resp. $A^{(4)}_{23}$).
\par Let $A'$ be the set $\left\{({\bm a},\dots,{\bm a}) \,\middle|\, {\bm a} \in A^{(4)}_{01} \cup A^{(4)}_{12} \cup A^{(4)}_{23} \right\}\backslash \{(1,\dots,1),(2,\dots,2)\}$. 
We have proven that $A^{(4)} \subset A'$ as above. 
A brief calculation shows that $\Delta({\bm a}) = 0$ for ${\bm a} \in A^{(4)}_{01} \cup A^{(4)}_{12} \cup A^{(4)}_{23}$ and hence $({\bm a},\dots,{\bm a}) \in A'$ gives a $\mathbb{Z}$-coloring of $D$. 
To see that every element of $A'$ actually defines a four-color coloring, it is sufficient to show that for any $(a_1,\dots,a_{qr}) \in A'$, $(-a_2 + 2a_1, \dots, -a_{qr} + 2a_1, a_1)$ is also a member of $A'$; this means that ``one twist preserves $A'$'', and then implies that there are only four colors in the resultant $\mathbb{Z}$-coloring. 
This is verified by the definitions of $A^{(4)}_{01}, A^{(4)}_{12}$, and $A^{(4)}_{23}$. 
Thus we have $A' \subset A^{(4)}$ as required.
\end{proof}

\appendix

\section{$\mathbb{Z}$-coloring of $n$-parallels of knots}\label{sec:n-parallel}

In this appendix, we give a proof of Theorem~\ref{parallel-thm}. 
In this proof, we use ideas of racks and quandles. For definitions of racks, quandles, and rack (quandle) colorings, see \cite{ElhamdadiNelson} for example. For a rack $R$, we denote the set of $R$-colorings of an oriented knot diagram $D$ by ${\rm Col}_R(D)$. For example, the set $\mathbb{Z}$ equipped with the binary operation $*$ defined by $a * b = 2b - a$ is a rack (it is in fact a quandle), and then ${\rm Col}_{\mathbb{Z}}(D)$ is the set of $\mathbb{Z}$-colorings as in Section \ref{sec:description}.

We denote the automorphism group of a rack $R=(R,*)$ by ${\rm Aut}(R,*)$. For each $a \in R$, the map $\bullet * a: R \ni x \mapsto x * a \in R$ is by definition is an automorphism of $R$, and then the subgroup of ${\rm Aut}(R,*)$ generated by $\bullet * a\;(a\in R)$ is called the \textit{inner automorphism group} and denoted by ${\rm Inn}(R,*)$. We say that $R$ is (\textit{algebraically}) \textit{connected} if the action of ${\rm Inn}(R,*)$ on $R$ is transitive.

Furthermore, we should recall the quandles and kink maps associated to racks. The \textit{associated quandle} $R_Q$ of a rack $R = (R, *)$ is the pair $(R,*_Q)$, where the binary operation $*_Q$ is defined by $x *_Q y = (x \mathbin{\bar{*}} x) * y$ $(x,y \in R)$, and the \textit{associated kink map} $\tau$ is defined by $\tau(x) = x * x$ $(x \in R)$. A brief calculation shows that $R_Q$ is a quandle and $\tau$ is a kink map of $R$ (see, e.g., \cite{AndruskiewitschGrana}), i.e., $\tau$ is an automorphism of $R$ and for any $x,y \in R$ we have $x * \tau(y) = x * y$; this implies that $\tau$ is also a kink map of $R_Q$. We remark that the quandle $R_Q$ is equal to $R$ as a set, and an automorphism of $R$ is also an automorphism of $R_Q$; we may regard ${\rm Aut}(R,*)$ and ${\rm Inn}(R,*)$ as subgroups of ${\rm Aut}(R_Q)$.

To show the theorem, we introduce a rack $\mathbb{Z}^n_R$ as follows (this is due to \cite{Naruse}): for ${\bm x} = (x_1, \dots, x_n), {\bm y} = (y_1,\dots,y_n) \in \mathbb{Z}^n,$ we put ${\bm x} *_R {\bm y} = {\bm z} = (z_1,\dots,z_n)$, where
\[ z_i = (((x_i * y_1) * y_2) \cdots ) * y_n. \]
By a brief calculation, we can check that $\mathbb{Z}^n_R = (\mathbb{Z}^n, *_R)$ is a rack. 

\par In the proof of Theorem \ref{parallel-thm} below, we examine the $\mathbb{Z}^n_R$-colorings of the knot diagram $D$. Since the fundamental rack of a framed knot is connected, the image of a coloring of $D$ by a rack $R$ is contained in a connected subrack of $R$; if $R$ is decomposed into the maximal connected subracks $R_{\lambda}\;(\lambda \in \Lambda)$ (for the existence and uniqueness of the decomposition, see, e.g., \cite{AndruskiewitschGrana}), we have ${\rm Col}_R(D) = \bigsqcup_{\lambda \in \Lambda} {\rm Col}_{R_{\lambda}}(D)$. The following lemma describes the subracks $R_{\lambda}$ when $R=\mathbb{Z}^n_R$.

\begin{lem}\label{max-lem}
Each maximal connected subrack of $\mathbb{Z}^n_R$ is a cyclic rack.
\end{lem}
\noindent Recall that a \textit{cyclic rack} $C_k$ ($k \in \mathbb{Z}_{\geq 0}$) is a cyclic group $\mathbb{Z}/k\mathbb{Z}$ with the binary operation $*$ defined by $a * b = a+1$ ($a,b \in \mathbb{Z}/k\mathbb{Z}$). A rack isomorphic to a cyclic rack is also called a cyclic rack.
\begin{proof}[Proof of Lemma \ref{max-lem}]
Let $\mathbb{Z}^n_Q = (\mathbb{Z}^n,*_Q)$ be the associated quandle of $\mathbb{Z}^n_R$ and $\tau$ the associated kink map. 
We first claim that each maximal connected subquandle of $\mathbb{Z}^n_Q$ is the trivial quandle of order $1$. In fact, a brief calculation shows that
\[ (x_i)_i *_Q (y_i)_i = \big( x_i + 2(-x_n + x_{n-1} - \cdots +(-1)^n x_1 + (-1)^{n+1} y_1 + \cdots + y_n) \big)_i,\]
for $(x_i)_i, (y_i)_i \in \mathbb{Z}^n_Q.$ This shows that the orbit of  $(x_i)_i \in \mathbb{Z}^n_Q$ under the action of ${\rm Inn}(\mathbb{Z}^n_Q)$ is included in $\{ (x_i + a)_i \mid a \in \mathbb{Z} \}.$ Furthermore, since
\[ (x_i + a)_i *_Q (x_i + b)_i = \big(x_i + a + (1-(-1)^n)(b-a) \big)_i = \big( x_i + (-1)^n a + (1-(-1)^n)b \big)_i, \]
the quandle $\{ (x_i + a)_i \mid a \in \mathbb{Z} \}$ is isomorphic to $\mathbb{Z}$ or a trivial quandle. In either case, each maximal connected subquandle is the trivial quandle of order $1$. The lemma follows from this claim and Lemma \ref{tau-lem} below.
\end{proof}

\begin{lem}\label{tau-lem}
Let $R = (R,*)$ be a rack, $R_Q = (R, *_Q)$ the associated quandle, and $\tau$ the associated kink map. We denote the maximal connected subrack of $R$ {\rm(}resp. $R_Q${\rm)} containing $x \in R$ {\rm(}resp. $R_Q${\rm)} by $\mathcal{M}_x$ {\rm(}resp. $\mathcal{M}^Q_x${\rm)}. Then we have $\mathcal{M}_x = \bigcup_{m \in \mathbb{Z}} \tau^m(\mathcal{M}^Q_x).$
\end{lem}
\begin{proof}
We set $\mathcal{M}'_x = \bigcup_{m \in \mathbb{Z}} \tau^m(\mathcal{M}^Q_x).$ We have
\begin{align*}
\tau^m(y) * \tau^n(z) &= ((\tau^m(y) * \tau^m(y)) \mathbin{\bar{*}} \tau ^m(y)) * z = \tau^{m+1}(y) *_Q z\\
&= \tau^{m+1}(y *_Q z)
\end{align*}
and similarly $\tau^m(y) \mathbin{*} \tau^n(z) = \tau^{m-1}(y \mathbin{\overline{*_Q}} z)$ for any $y,z \in R$ (especially for $y,z \in \mathcal{M}^Q_x$) and $m,n \in \mathbb{Z}$, $\mathcal{M}'_x$ is a connected subrack of $R$; hence we have $\mathcal{M}'_x \subset \mathcal{M}_x$. Furthermore, by the definitions of $*_Q$ and $\tau$, ${\rm Inn}(\mathcal{M}_x,*)$ is contained in the subgroup of ${\rm Aut}(\mathcal{M}_x,*_Q)$ generated by ${\rm Inn}(\mathcal{M}_x, *_Q)$ and $\tau$. Since $\tau$ is central in this subgroup (recall that $\tau$ is a kink map of $(\mathcal{M}_x,*_Q)$, we have $\mathcal{M}_x = \bigcup_m \tau^m(\mathcal{M}'^Q_x)$, where we set $\mathcal{M}'^Q_x = {\rm Inn}(\mathcal{M}_x, *_Q)\cdot x$, i.e., the orbit of $x$ under the action of ${\rm Inn}(\mathcal{M}_x, *_Q)$. Here, the inner automorphism group ${\rm Inn}(\mathcal{M}_x, *_Q)$ is by definition generated by $\bullet *_Q y$ ($y \in \mathcal{M}_x$), but for each $y \in \mathcal{M}_x = \bigcup_m \tau^m(\mathcal{M}'^Q_x)$ there exists $y' \in \mathcal{M}'^Q_x$ such that $\bullet *_Q y = \bullet *_Q y'$; in fact, if $y \in \tau^m(\mathcal{M}'^Q_x)$, we can take $\tau^{-m}(y)$ as $y'$. Thus, $\mathcal{M}'^Q_x = {\rm Inn}(\mathcal{M}'^Q_x, *_Q) \cdot x$, which implies that $\mathcal{M}'^Q_x$ is a connected subquandle. Then we have $\mathcal{M}'^Q_x \subset \mathcal{M}^Q_x$ and hence $\mathcal{M}_x = \bigcup_m \tau^m(\mathcal{M}'^Q_x) \subset \bigcup_m \tau^m(\mathcal{M}^Q_x) = \mathcal{M}'_x$, as required.
\end{proof}

\begin{proof}[Proof of Theorem \ref{parallel-thm}]
Given a $\mathbb{Z}^n_R$-coloring $\mathcal{C}^{(n)}= (\mathcal{C}_1,\dots,\mathcal{C}_n): \{\text{arcs}\} \to \mathbb{Z}^n_R$ on $D$, we put colors $\mathcal{C}_1(\alpha),\dots,\mathcal{C}_n(\alpha)$ to the $n$ arcs of $D^{(n)}$ corresponding to each arc $\alpha$ of $D$. This is uniquely extended to be a whole coloring on $D^{(n)}$ (see Figure \ref{bij-fig}). Conversely, a $\mathbb{Z}$-coloring $\mathcal{C}$ of $D^{(n)}$ defines a $\mathbb{Z}^n_R$ coloring: we associate $(\mathcal{C}(\alpha_1),\dots,\mathcal{C}(\alpha_n))$ to an arc $\alpha$ of $D$, where $\alpha_1,\dots,\alpha_n$ are the $n$ arcs of $D^{(n)}$ corresponding to $\alpha$; by the definition of $*_R$, this defines a $\mathbb{Z}^n_R$-coloring. Thus, we have a bijection between ${\rm Col}_{\mathbb{Z}^n}(D)$ and ${\rm Col}_{\mathbb{Z}}(D^{(n)})$. In the following, we identify these two sets by this bijection.

%
\begin{figure}[t]
\centering
\includegraphics{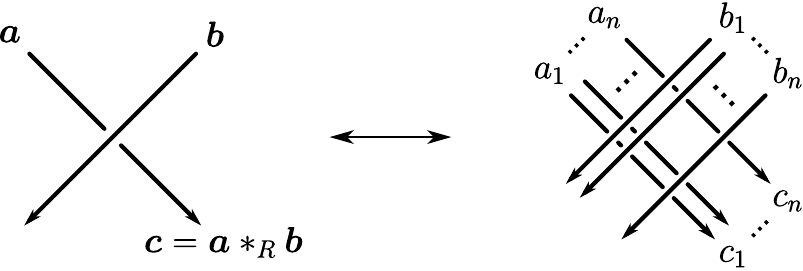}
\caption{A bijection between ${\rm Col}_{\mathbb{Z}^n}(D)$ and ${\rm Col}_{\mathbb{Z}}(D^{(n)})$}\label{bij-fig}
\end{figure}
\par By Lemma \ref{max-lem}, any $\mathbb{Z}^n_R$-coloring on $D$ is a coloring by a subrack, which is a cyclic rack. Since we only have to consider cyclic-rack colorings, a given color ${\bm a} \in \mathbb{Z}^n_R$ on the fixed arc $\gamma$ uniquely determines the colors of the other arcs successively, and it defines a whole coloring if and only if they accords when we go back to $\gamma$, i.e., $\tau^w({\bm a}) = {\bm a},$ where $\tau$ is the associated kink map. Thus, $r$ is injective and ${\rm Im}\,r = \{{\bm a} \in \mathbb{Z}^n \mid \tau^w({\bm a}) = {\bm a}\}.$
\par By a concrete calculation we find that
\[ \tau({\bm a}) = \big( (-1)^n a_i + 2(a_n - a_{n-1} + \cdots + (-1)^{n-1} a_1) \big)_i. \]
If $n$ is even, this shows that
\[ \tau^w({\bm a}) = \big( a_i + 2w(a_n - a_{n-1} + \cdots - a_1) \big)_i.\]
Then $\tau^w({\bm a}) = {\bm a}$ if and only if $w(a_1 - a_2 + \cdots - a_n) = 0$, as required. Next, suppose that $n$ is odd. In this case, a brief calculation shows that $\tau^2 = {\rm id}_{\mathbb{Z}^n}$. Then, if $w$ is even, the condition $\tau^w({\bm a}) = {\bm a}$ is always satisfied. If $w$ is odd, $\tau^w({\bm a}) \;(= \tau({\bm a}))$ equals ${\bm a}$ if and only if $a_i = - a_i + 2\Delta$ for $i = 1, \dots, n$, where $\Delta = a_n - a_{n-1} + \cdots + a_1$. This implies that $a_i = \Delta$ for any $i$ and hence $a_1 = \cdots = a_n$; this concludes the theorem.
\end{proof}

\end{document}